\documentclass[11pt]{article}

\usepackage{amsmath,amssymb,amsthm,amsfonts,amstext,amsbsy,amscd,color}
\usepackage{graphicx}

\RequirePackage[colorlinks=true,citecolor=blue,urlcolor=blue]{hyperref}

\let\tilde=\widetilde

\newcommand{\field}[1]{\mathbb{#1}}
\newcommand{\R}{\field{R}}

\newcommand{\PP}{\field{P}}

\def\der^#1_#2{\frac{\partial^{#1}}{\partial {#2}^{#1}}}

\theoremstyle{plain}
\newtheorem{theorem}{Theorem}

\newtheorem{lemma}{Lemma}
\newtheorem{prop}{Proposition}
\newtheorem{assumption}{Assumption}{\bf}{\rm}
\theoremstyle{remark}
\newtheorem{remark}{Remark}
{\smallskip}%

\newcommand{\E}{\field{E}}

\begin{document}


\title{Statistical estimation of a growth-fragmentation model observed on a genealogical tree} 


\author{Marie Doumic\thanks{INRIA Rocquencourt, projet BANG, Domaine de Voluceau, BP 105, 781153 Rocquencourt, France and Laboratoire J.L. Lions, CNRS-UMR 7598, 4 place Jussieu 75005 Paris, France {\tt email}: marie.doumic-jauffret@inria.fr}
\and Marc Hoffmann\thanks{Universit\'e Paris-Dauphine, CNRS-UMR 7534,
Place du mar\'echal De Lattre de Tassigny
75775 Paris Cedex 16, France. Corresponding author. {\tt email}: hoffmann@ceremade.dauphine.fr} 
 \and Nathalie Krell\thanks{Universit\'e de Rennes 1, CNRS-UMR 6625, Campus de Beaulieu, 35042 Rennes Cedex, France. {\tt email}: nathalie.krell@univ-rennes1.fr}
\and Lydia Robert\thanks{INRA, UMR1319, Micalis, F-78350 Jouy-en-Josas, France.}\;\,\thanks{AgroParisTech, UMR Micalis, F-78350 Jouy-en-Josas, France.
  {\tt email}: lydia.robert@upmc.fr}}
\date{}
\maketitle
\begin{abstract}
We raise the issue of estimating the division rate for a growing and dividing population 
modelled  by a piecewise deterministic Markov branching tree. Such models have broad applications, ranging from TCP/IP window size protocol to bacterial growth. Here, the individuals split into two offsprings at a division rate $B(x)$ that depends on their size $x$, whereas their size  grow exponentially in time, at a rate that exhibits variability. 
The mean empirical measure of the model satisfies a growth-fragmentation type equation, and we bridge the deterministic and probabilistic viewpoints. 
We then construct a nonparametric estimator of the division rate $B(x)$ based on the observation of the population over different sampling schemes of size $n$ on the genealogical tree. Our estimator nearly achieves the rate $n^{-s/(2s+1)}$ in squared-loss error asymptotically,  generalizing and improving on the rate $n^{-s/(2s+3)}$ obtained in \cite{DHRR, DPZ}  through indirect observation schemes.  Our method is consistently tested numerically and implemented on {\it Escherichia coli} data, which demonstrates its major interest for practical applications. 
\end{abstract}




{\small 
}
\section{Introduction}
\subsection{Size-structured models and their inference} \label{size-strutured models}
Growth-fragmentation models and structured population equations  describe the temporal evolution of a population characterised by state variables such as age, size, growth, maturity, protein content and so on - see \cite{Metz,Pe} and the references therein. 
This field continues to grow in interest as its applications  appear to be wider and wider, ranging from the internet TCP/IP window size protocol \cite{Baccelli} to neuronal activity \cite{PPS3}, protein polymerization \cite{Engler}, cell division cycle \cite{Banks1}, phase transitions \cite{NP4}, parasite proliferation \cite{Bansaye} etc. 

In order to  quantitatively fit experimental observations and thus validate the relevance of the models, developing new and well-adapted statistical methods appears to be one of the major challenge for the coming years.
A paradigmatic example, which can serve both as a toy model and a proof of concept for the methodology we develop here,  is given by the growth-fragmentation or size-structured cell division equation \cite{PR}. When applied to the evolution of a bacterial population it reads
\begin{equation} \label{transport-fragmentation eq}
\left\{
\begin{array}{l}
\partial_t n(t,x) + \tau \partial_x\big(x\, n(t,x)\big) + B(x)n(t,x) = 4B(2x)n(t,2x) \\ \\
n(0,x)= n^{(0)}(x), x \geq 0,
\end{array}
\right.
\end{equation}
and it quantifies the concentration $n(t,x)$ of individuals (cells) having size $x$ (the state variable) at time $t$.
A common stochastic mechanism for every single  cell is attached to Equation~\eqref{transport-fragmentation eq}\-:
\begin{enumerate}
\item The size $x = x(t)$ of a cell at time $t$ evolves exponentially according to the deterministic evolution $dx(t)=\tau x(t)dt$, where $\tau >0$ is
the growth rate of each cell, that quantifies their ability to ingest a common nutrient.
\item Each cell splits into two offsprings according to a division rate $B(x)$ that depends on its current size $x$.
\item At division, a cell of size $x$ gives birth to two offsprings of size $x/2$ each, what is called \emph{binary fission}.
\end{enumerate}
Model \eqref{transport-fragmentation eq} is thus entirely determined by the parameters $\big(\tau, B(x), x \in [0,\infty)\big)$. Typically, the growth rate $\tau$ is assumed to be known or guessed \cite{DMZ}, and thus inference about \eqref{transport-fragmentation eq} mainly concerns the estimation of the division rate $B(x)$ that has to be taken from a nonparametric perspective.\\ 

By use of the general relative entropy principle, Michel, Mischler and Perthame showed that the approximation $n(t,x) e^{-\lambda_0 t} \approx N(x)$ is valid \cite{MMP}, with $\lambda_0 >0$, and where $(\lambda_0,N)$ is the dominant eigenpair related to the corresponding eigenvalue problem, see  \cite{M1, Pe, DG,LP,CCM, BCG}.
The ``stationary" density $N(x)$ of typical cells after some time has elapsed enables to recover $(B(x), x \in {\mathcal D}\big)$ for a compact ${\mathcal D} \subset (0,\infty)$ by means of the regularisation of an inverse problem of ill-posedness degree 1. From a deterministic perspective, this is carried out in \cite{PZ,DPZ, DT}.
From a statistical inference perspective, if an $n$-sample of the distribution $N(x)$ is observed and if $B(x)$ has smoothness $s>0$ in a Sobolev sense, it is proved in \cite{DHRR} that $B(x)$ can be recovered in squared-error loss over compact sets with a rate of convergence $n^{-s/(2s+3)}$.
Both deterministic and stochastic methodology of \cite{DPZ} and \cite{DHRR} are motivated by experimental designs and data such as in \cite{Kubitschek, DMZ}. However, they do not take into account the following two important aspects:
\begin{itemize}
\item Bacterial growth exhibits variations in the individual growth rate $\tau$ as demonstrated for instance in \cite{Sturm}. One would like to incorporate variability in the growth rate within the system at the level of a single cell. This requires to modify Model~\eqref{transport-fragmentation eq}.
\item Recent evolution of experimental technology enables to track the whole genealogy of cell populations (along prescribed lines of descendants for instance), affording the observation of other state variables such as size at division, lifetime of a single individual and so on \cite{Robert}. 
Making the best possible use of such measures is of great potential impact, and needs a complementary approach. 
\end{itemize}
The availability of observation schemes at the level of cell individuals suggests an enhancement of the statistical inference of $\big(B(x), x \in {\mathcal D}\big)$,  enabling to improve on the rates of convergence obtained by indirect measurements such as in \cite{DHRR, DPZ}. 
This is the purpose of the present paper. We focus on bacterial growth, for which we  apply our method on experimental observations. This serves as a proof of concept for the relevance of our modelling and statistical methodology, which could adapt to other application fields and growth-fragmentation types.

\subsection{Results of the paper}
\subsubsection*{Statistical setting}
Let
$${\mathcal U} = \bigcup_{k=0}^\infty \{0,1\}^k$$ 
denote the binary genealogical tree (with $\{0,1\}^0:=\{\emptyset\}$). We identify each node $u \in {\mathcal U}$ with a cell that has a size at birth $\xi_u$ and a lifetime $\zeta_u$. 
In the paper, we consider the problem of estimating $\big(B(x), x \in [0,\infty)\big)$ over compact sets of $(0,\infty)$. Our inference procedure is based on the observation of 
\begin{equation} \label{micro dataset}
\big((\xi_u, \zeta_u), u \in {\mathcal U}_n\big).
\end{equation}
where ${\mathcal U}_n \subset {\mathcal U}$ denotes a connected subset of size $n$ containing the root $u  = \emptyset$.
Asymptotics are taken as $n\rightarrow \infty$. Two important observation schemes are considered: the sparse tree case, when we follow the system along a given branch with $n$ individuals, and the full tree case, where we follow the evolution of the whole binary tree up to the $N_n$-th generation, with $N_n \approx \log_2 n$.
In this setting, we are able to generalise Model \eqref{transport-fragmentation eq} and allow  the growth rate $\tau$ to vary with each cell $u \in {\mathcal U}$.
We assume that a given cell $u$ has a random growth rate $\tau_u=v \in {\mathcal E} \subset (0,\infty)$ (later constrained to live on a compact set). Moreover, this value $v$ is inherited from the growth rate $v'$ of its parent according to a distribution $\rho(v',dv)$.   Since a cell splits into two offsprings of the same size, letting $u^-$ denote the parent of $u$, we have the fundamental relationship 
\begin{equation} \label{fundamental}
2\,\xi_u = \xi_{u^-}\exp\big(\tau_{u^-} \zeta_{u^-}\big)
\end{equation}
that enables to recover the growth rate $\tau_u$ of each individual in ${\mathcal U}_n$ since ${\mathcal U}_n$ is connected by assumption, possibly leaving out the last generation of observed individuals, but this has asymptotically no effect on a large sample size approach.

\subsubsection*{Variability in growth rate}

In the case where the growth rate can vary for each cell, the density $n(t,x)$ of cells of size $x$ at time $t$ does not follow Eq.~\eqref{transport-fragmentation eq} anymore and an extended framework needs to be considered. To that end, we structure the system with an additional variable $\tau_u=v,$ which represents the growth rate and depends on  each individual cell $u$. We construct in Section \ref{microscopic model} a branching Markov chain $\big((\xi_u, \tau_u), u \in {\mathcal U}\big)$ that incorporates variability for the growth rate in the mechanism described in Section \ref{size-strutured models}. 
Equivalently to the genealogical tree, the system may be described in continuous time by a piecewise deterministic Markov process
$$\big(X(t),V(t)\big) = \Big(\big(X_1(t),V_1(t)\big),\big(X_2(t),V_2(t)\big),\ldots\Big),$$
which models the process of sizes and growth rates of the living particles in the system at time $t$, with value in $\bigcup_{k=0}^\infty {\mathcal S}^k$, where ${\mathcal S} = [0,\infty)\times {\mathcal E}$ is the state space of size times growth rate.  Stochastic systems of this kind that correspond to branching Markov chains are fairly well known, both from a theoretical angle and in applications; a selected list of contributions is \cite{BDMV, Cloez, MT} and the references therein.

By fragmentation techniques inspired by Bertoin \cite{bertoin}, see also Haas \cite{Haas}, we relate the process $(X,V)$ to a growth-fragmentation equation as follows. Define
$$\langle n(t,\cdot), \varphi\rangle = \E\Big[\sum_{i = 1}^\infty \varphi\big(X_i(t), V_i(t)\big)\Big]$$
as the expectation of the empirical measure of the process $(X,V)$ over smooth test functions defined on ${\mathcal S}$. We prove in Theorem \ref{sol transport general} that, under appropriate regularity conditions, the measure $n(t,\cdot)$ that we identify with the temporal evolution of the density $n(t,x,v)$ of cells having size $x$ and growth rate $v$ at time $t$ is governed (in a weak sense\footnote{
For every $t\geq 0$, we actually have a Radon measure $n(t,dx ,dv)$ on ${\mathcal S}=[0,\infty)\times {\mathcal E}$:
If $\varphi(x,v)$ is a function defined on ${\mathcal S}$, we define $\langle n(t,\cdot), \varphi\rangle = \int_{{\mathcal S}}\varphi(x,v)n(t,dx, dv)$ whenever the integral is meaningful. 
Thus \eqref{transport variabilite} has the following sense: for every sufficiently smooth test function $\varphi$ with compact support in ${\mathcal E}$, we have
\begin{align*}
&\int_{{\mathcal S}}\partial_t n(t,dx,dv)\varphi(x,v)-\,vx n(t,dx, dv)\partial_x\varphi(x,v)+B(x)n(t,dx,dv)\varphi(x,v) \\
 = & \;4\int_{{\mathcal S}}\big(B(2x)\int_{{\mathcal E}} \rho(v',dv)n(t,2dx,dv')\varphi(x,v)\big).
 \end{align*}
 }) by 
\begin{equation} \label{transport variabilite}
\left\{
\begin{array}{rl}
& \partial_t n(t,x,v) + v\, \partial_x\big(x\, n(t,x, v)\big)+B(x)n(t,x,v) \\ \\
 & =\; 4B(2x)\int_{{\mathcal E}} \rho(v',v)n(t,2x,dv'), \\ \\
 & n(0,x,v)= n^{(0)}(x,v), x \geq 0.
\end{array}
\right.
\end{equation}
If we assume a constant growth rate $\tau>0$, we then take $\rho(v', dv) = \delta_{\tau}(dv)$ (where $\delta$ denotes the Dirac mass) and we retrieve the standard growth-fragmentation equation \eqref{transport-fragmentation eq}. The proof of Theorem \ref{sol transport general} is obtained via a so-called many-to-one formula, established in Proposition \ref{many-to-one} in Section \ref{a many-to-one formula via a tagged branch}. Indeed, thanks to the branching property of the system, it is possible to relate the behaviour of additive functionals like the mean empirical measure to the behaviour of a so-called tagged cell (like a tagged fragment in fragmentation process), that consists in following the behaviour of a single line of descendants along a branch where each node is picked at random, according to a uniform distribution. This approach, inspired by fragmentation techniques, is quite specific to our model and enables to obtain a relatively direct proof of Theorem \ref{transport variabilite}. 
\subsubsection*{Nonparametric estimation of the growth rate}
In Section \ref{statistical analysis} we take over the problem of estimating $(B(x), x \in {\mathcal D})$ for some compact ${\mathcal D} \subset (0,\infty)$. We assume we have data of the form \eqref{micro dataset}, and that the mean evolution of the system is governed by \eqref{transport variabilite}. The growth rate kernel $\rho$ is unknown and treated as a nuisance parameter.
A fundamental object is the transition kernel
$$
{\mathcal P}_{B}(\boldsymbol{x},d\boldsymbol{x'}) = \PP\big((\xi_u,\tau_u)\in d\boldsymbol{x'}\big|\,(\xi_{u^-},\tau_{u^-})=\boldsymbol{x}\big)
$$
of the size and growth rate distribution $(\xi_u,\tau_u)$ at the birth of a descendant $u \in {\mathcal U}$, given the size of birth and growth rate of its parent $(\xi_{u^-},\tau_{u^-})$. We define in Section \ref{upper rates} a class of division rates and growth rate kernels such that if $(B,\rho)$ belongs to this class,
then the transition ${\mathcal P}_{B}$ is geometrically ergodic and has a unique invariant measure $\nu_B(d\boldsymbol{x}) = \nu_B(x,dv)dx$. 
From the invariant measure equation
$$\nu_B{\mathcal P}_B=\nu_B$$
we obtain in Proposition \ref{rep B via inv} the explicit representation
\begin{equation} \label{rep B anticipate}
B(x)=\frac{x}{2}\frac{\nu_B(x/2)}{\E_{\nu_B}\Big[\frac{1}{\tau_{u^-}}{\bf 1}_{\{\xi_{u^-} \leq x,\;\xi_u \geq x/2\}}\Big]}.
\end{equation}
where $\nu_B(x) = \int_{{\mathcal E}}\nu_B(x,dv)$ denotes the first marginal of the invariant distribution $\nu_B$. A strategy for constructing and estimator $B$ consists in replacing the right-hand size of \eqref{rep B anticipate} by its empirical counterpart, the numerator being estimated via a kernel estimator of the first maginal of $\nu_B(d\boldsymbol{x})$. Under local H\" older smoothness assumption on $B$ of order $s>0$, we prove in Theorem \ref{upper bound} that for a suitable choice of bandwidth in the estimation of the invariant density, our estimator achieves the rate $n^{-s/(2s+1)}$ in squared-error loss over appropriate compact sets ${\mathcal D} \subset (0,\infty)$, up to an inessential logarithmic term when the full tree observation scheme is considered. We see in particular that we improve on the rate obtained in \cite{DHRR}. Our result quantifies the improvement obtained when estimating $B(x)$ from data $\big((\xi_u, \zeta_u), u \in {\mathcal U}_n\big)$, as opposed to overall measurements of the system after some time has elapsed as in \cite{DHRR}. We provide a quantitative argument based on the analysis of a PDE that explains the reduction of ill-posedness achieved by our method over \cite{DHRR} in Section \ref{num}.
\\

In order to obtain the upper bound of Theorem \ref{upper bound}, a major technical difficulty is that we need to establish uniform rates of convergence of the empirical counterparts to their limits in the numerator and denominator of \eqref{rep B anticipate} when the data are spread along a binary tree. This can be done via covariance inequalities that exploit the fact that the transition ${\mathcal P}_{B}$ is geometrically ergodic (Proposition \ref{prop transition}) using standard Markov techniques, see \cite{MT, B}. The associated chain is however not reversible, and this yields an extraneous difficulty: the decay of the correlations between $\varphi(\xi_u, \tau_u)$ and $\varphi(\xi_v, \tau_v)$ for two nodes $u,v\in {\mathcal U}_n$ are expressed in terms of the sup-norm of $\varphi$, whenever $|\varphi({\boldsymbol{x}})| \leq {\mathbb V}(\boldsymbol{x})$ is dominated by a certain Lyapunov function 
${\mathbb V}$ for the transition ${\mathcal P}_{B}$. However, the typical functions $\varphi$ we use are kernels that depend on $n$ and that are not uniformly bounded in sup-norm as $n \rightarrow \infty$. This partly explains the relative length of the technical Sections \ref{covariance inequalities} and \ref{rate of convergence for the empirical measure}.

\subsection{Organisation of the paper}
In Section \ref{microscopic model}, we construct the model $\big((\xi_u, \tau_u), u \in {\mathcal U}\big)$ of sizes and growth rates 
of the cells as a Markov chain along the genealogical tree. The discrete model can be embedded into a continuous time piecewise deterministic Markov process $(X,V)$
of sizes and growth rates of the cells present at any time within the system. In Theorem \ref{sol transport general} we explicit the relation between the mean empirical measure of $(X,V)$ and the growth-fragmentation type equation \ref{transport variabilite}. In Section \ref{statistical analysis}, we explicitly construct an estimator $\widehat B_n$ of $B$ by means of the representation given by \eqref{rep B anticipate} in Section \ref{estimation of the division rate}. Two observation schemes are considered and discussed in Section \ref{two observation schemes}, whether we consider data along a single branch (the sparse tree case) or along the whole genealogy (the full tree case). The specific assumptions and the class of admissible division rates $B$ and growth rate kernels $\rho$ are discussed in Section \ref{upper rates}, and an upper bound for $\widehat B_n$ in squared-error loss is given in our main Theorem \ref{upper bound}. Section \ref{numerical implementation} shows and discusses the numerical implementation of our method on simulated data. In particular, ignoring the variability in the reconstruction dramatically deterioriates the accuracy of estimation of $B$. We also explain from a deterministic point perspective the rate improvement of our method compared with \cite{DHRR} by means of a PDE analysis argument in Section \ref{num}. The parameters are inspired from real data experiments on {\it Escherichia coli} cell cultures. Section \ref{proofs} is devoted to the proofs. 
\section{A Markov model on a tree} \label{microscopic model}
\subsection{The genealogical construction}
Recall that 
${\mathcal U} := \bigcup_{n=0}^\infty \{0,1\}^n$ 
(with $\{0,1\}^0:=\{\emptyset\}$) denotes the infinite binary genealogical tree. Each node $u \in {\mathcal U}$ is identified with a cell of the population and has a mark 
$$(\xi_u, b_u, \zeta_u, \tau_u),$$
where $\xi_u$ is the size at birth, $\tau_u$ the growth rate, $b_u$ the birthtime and $\zeta_u$ the lifetime of $u$. The evolution $\big(\xi_t^{u}, t \in [b_u,b_u+\zeta_u)\big)$ of the size of $u$ during its lifetime is governed by 
\begin{equation} \label{piecewise deterministic}
\xi_t^u=\xi_u \exp\big(\tau_u (t-b_u)\big)\;\;\text{for}\;\;t\in [b_u,b_u+\zeta_u).
\end{equation}
Each cell splits into two offsprings of the same size according to a division rate $B(x)$ for $x\in (0,\infty)$. Equivalently 
\begin{equation} \label{def div rate}
\PP\big(\zeta_{u}\in [t,t+dt]\,\big|\,\zeta_{u}\geq t,\xi_{u}=x, \tau_u=v\big) = B\big(x\exp(v t)\big)dt.
\end{equation}
At division, a cell splits into two offsprings of the same size. If $u^-$ denotes the parent of $u$, we thus have 
\begin{equation} \label{fundamental}
2\,\xi_u = \xi_{u^-}\exp\big(\tau_{u^-} \zeta_{u^-}\big)
\end{equation}
Finally, the growth rate $\tau_u$ of $u$ is inherited from its parent $\tau_{u^-}$ according to a Markov kernel
\begin{equation} \label{markov heritage}
\rho(v,dv')=\PP(\tau_u\in dv'\,|\,\tau_{u^-}=v),
\end{equation}
where $v >0$ and $\rho(v,dv')$ is a probability measure on $(0,\infty)$ for each $v>0$.
Eq. \eqref{piecewise deterministic}, \eqref{def div rate}, \eqref{fundamental} and \eqref{markov heritage} completely determine the dynamics of the  model $\big((\xi_u, \tau_u), u \in {\mathcal U}\big)$, as a Markov chain on a tree, given an additional initial condition $(\xi_\emptyset, \tau_\emptyset)$ on the root. 
The chain is embedded into a piecewise deterministic continuous Markov process thanks to \eqref{piecewise deterministic} by setting
$$(\xi^u_t, \tau_t^u) =\big(\xi_u \exp\big(\tau_u( t-b_u)\big), \tau_u\big)\;\;\text{for}\;\;t \in [b_u, b_u+\zeta_u)$$
and $(0,0)$ otherwise.
Define
$$\big(X(t),V(t)\big) = \Big(\big(X_1(t),V_1(t)\big),\big(X_2(t),V_2(t)\big),\ldots\Big)$$
as the process of sizes and growth rates of the living particles in the system at time $t$. 
We have an identity between point measures
\begin{equation} \label{identity point measures}
\sum_{i = 1}^\infty {\bf 1}_{\{X_i(t)>0\}}\delta_{(X_i(t),V_i(t))} = \sum_{u \in {\mathcal U}}{\bf 1}_{\{b_u \leq t < b_u +\zeta_u\}}\delta_{(\xi_t^u,\tau_t^u)}
\end{equation}
where $\delta$ denotes the Dirac mass.
In the sequel, the following basic assumption is in force.
\begin{assumption}[Basic assumption on $B$ and $\rho$] \label{basic assumption}
The division rate $x \leadsto B(x)$ is continuous. We have $B(0)=0$ and 
$\int^\infty x^{-1}B(x)dx=\infty$. The Markov kernel $\rho(v,dv')$ is defined
on a compact set ${\mathcal E} \subset (0,\infty)$.
\end{assumption}
\begin{prop} \label{existence processus}
Work under Assumption \ref{basic assumption}. The law of 
$$\big((X(t),V(t)), t \geq 0\big)\;\;\text{or}\;\;\big((\xi_u, \tau_u), u \in {\mathcal U}\big)\;\;\text{or}\;\;\big((\xi_t^u, \zeta_t^u), t \geq 0, u \in {\mathcal U}\big)$$ is well-defined on an appropriate probability space with almost-surely no accumulation of jumps.
\end{prop}
If $\mu$ is a probability measure on the state space ${\mathcal S} = [0,\infty)\times {\mathcal E}$, we shall denote indifferently by $\PP_\mu$ the law of any of the three processes above 
where the root $(\xi_{\emptyset}, \tau_{\emptyset})$ has distribution $\mu$. The construction is classical (see for instance \cite{bertoin} and the references therein) and is outlined in Appendix \ref{construction of the chain}.
\subsection{Behaviour of the mean empirical measure}
Denote by ${\mathcal C}_0^1({\mathcal S})$ the set of real-valued test functions with compact support in the interior of ${\mathcal S}$. 
\begin{theorem}[Behaviour of the mean empirical measure] \label{sol transport general}
Work under Assumption \ref{basic assumption}. Let $\mu$ be a probability distribution on ${\mathcal S}$. Define the distribution $n(t,dx,dv)$ by
$$\langle n(t,\cdot),\varphi \rangle = \E_{\mu}\Big[\sum_{i = 1}^\infty\varphi\big(X_i(t), V_i(t)\big)\Big]\;\;\text{for every}\;\;\varphi \in {\mathcal C}^1_0({\mathcal S}).$$
Then $n(t,\cdot)$ solves (in a weak sense)
$$
\left\{
\begin{array}{rl}
& \partial_t n(t,x,v) + v\, \partial_x\big(x n(t,x, v)\big)+B(x)n(t,x,v) \\ \\
 & =\; 4B(2x)\int_{{\mathcal E}} \rho(v',v)n(t,2x,dv'), \\ \\
 & n(0,x,v)= n^{(0)}(x,v), x \geq 0.
\end{array}
\right.
$$
with initial condition $n^{(0)}(dx,dv) = \mu(dx,dv)$. 
\end{theorem}
Theorem \ref{sol transport general} somehow legitimates our methodology: by enabling each cell to have its own growth rate and by building-up new statistical estimators in this context, we still have a translation in terms of the approach in \cite{DPZ}. In particular, we will be able to compare our estimation results with \cite{DHRR}. Our proof is based on fragmentation techniques, inspired by Bertoin \cite{bertoin} and Haas \cite{Haas}. Alternative approaches to the same kind of questions include the probabilistic studies of Chauvin {\it et al.} \cite{CRW}, Bansaye {\it et al.} \cite{BDMV} or  Harris and Roberts \cite{HR} and references therein. 
\section{Statistical estimation of the division rate} \label{statistical analysis}
\subsection{Two observation schemes} \label{two observation schemes}
Let ${\mathcal U}_n \subset {\mathcal U}$ denote a subset of size $n$ of connected nodes: if $u$ belongs to ${\mathcal U}_n$, so does its parent $u^-$. We look for a nonparametric estimator of the division rate  
$$y\leadsto \widehat B_n(y) = \widehat B_n\big(y,(\xi_u, \tau_u), u \in {\mathcal U}_n)\big)\;\;\text{for}\;\;y \in (0,\infty)$$ 
Statistical inference is based on the observation scheme 
$$\big((\xi_u, \tau_u), u \in {\mathcal U}_n \big)$$ 
and asymptotic study is undertaken as the population size of the sample $n \rightarrow \infty$. We are interested in two specific observation schemes.
\begin{proof}[The full tree case]  We observe every pair $(\xi_u,\tau_u)$ over the first $N_n$ generations of the tree:
$${\mathcal U}_n = \{u \in {\mathcal U},\;\;|u| \leq N_n\}$$
with the notation 
$|u|=n$ if $u = (u_0, u_1,\ldots, u_n) \in {\mathcal U}$, and $N_n$ is chosen such that that $2^{N_n}$ has order $n$.
\end{proof}
\begin{proof}[The sparse tree case] We follow the first $n$ offsprings of a single cell, along a fixed line of descendants. This means that for some $u\in {\mathcal U}$ with $|u|=n$, we observe every size $\xi_u$ and growth rate $\tau_u$ of each node $(u_0)$, $(u_0,u_1)$, $(u_0,u_1,u_2)$ and so on up to a final node $u=(u_0,u_1,\ldots, u_{n})$.
\end{proof}
\begin{remark} \label{premiere remarque} For every $n \geq 1$, we tacitly assume that there exists a (random) time $T_n<\infty$ almost surely, such that for $t \geq T_n$, the observation scheme ${\mathcal U}_n$ is well-defined. This is a consequence of the behaviour of $B$ near infinity that we impose later on in \eqref{poly control} below.
\end{remark}
\subsection{Estimation of the division rate} \label{estimation of the division rate}
\subsubsection*{Identification of the division rate}
We denote by $\boldsymbol{x} = (x,v)$ an element of the state space ${\mathcal S} = [0,\infty)\times {\mathcal E}$. Introduce the transition kernel 
$$
{\mathcal P}_{B}(\boldsymbol{x},d\boldsymbol{x'}) = \PP\big((\xi_u,\tau_u)\in d\boldsymbol{x'}\big|\,(\xi_{u^-},\tau_{u^-})=\boldsymbol{x}\big)
$$
of the size and growth rate distribution $(\xi_u,\tau_u)$ at the birth of a descendant $u \in {\mathcal U}$, given the size at birth and growth rate of its parent $(\xi_{u^-},\tau_{u^-})$.
From \eqref{def div rate}, we infer that $\PP(\zeta_{u^-}\in dt\,\big|\,\xi_{u^-}=x, \tau_{u^-}=v)$ is equal to
$$
B\big(x\exp(v t)\big)\exp\Big(-\int_0^tB\big(x\exp(v s)\big)ds\Big)dt.
$$
Using formula \eqref{fundamental}, by a simple change of variables
$$\PP\big(\xi_u\in dx'\big|\,\xi_{u^-}=x, \tau_{u^-}=v\big)
=\frac{B(2x')}{v x'}{\bf 1}_{\{x' \geq x/2\}}\exp\big(-\int_{x/2}^{x'} \tfrac{B(2s)}{vs}ds\big)dx'.
$$
Incorporating \eqref{markov heritage},  we obtain an explicit formula for
$$
{\mathcal P}_{B}(\boldsymbol{x},d\boldsymbol{x'})  ={\mathcal P}_B\big((x,v), x', dv')dx',$$
with
\begin{equation} 
{\mathcal P}_B\big((x,v), x', dv')  = \frac{B(2x')}{v x'}{\bf 1}_{\{x' \geq x/2\}}\exp\big(-\int_{x/2}^{x'} \tfrac{B(2s)}{v s}ds\big) \rho(v,dv').  \label{density explicit}
\end{equation}
Assume further that ${\mathcal P}_B$ admits an invariant probability measure $\nu_B(d\boldsymbol{x})$, {\it i.e.} a solution to 
\begin{equation} \label{def mesure invariante}
\nu_B{\mathcal P}_B=\nu_B,
\end{equation}
where 
$$\mu{\mathcal P}_B(d\boldsymbol{y})=\int_{{\mathcal S}}\mu(d\boldsymbol{x}){\mathcal P}_B(\boldsymbol{x},d\boldsymbol{y})$$
denotes the left action of positive measures $\mu(d\boldsymbol{x})$ on $\mathcal{S}$ for the transition ${\mathcal P}_B$. 
\begin{prop} \label{rep B via inv}
Work under Assumption \ref{basic assumption}. If ${\mathcal P}_B$ admits an invariant probability measure $\nu_B$ of the form $\nu_B(d\boldsymbol{x}) = \nu_B(x,dv)dx$ then we have
\begin{equation} \label{representation cle}
\nu_B(y) =  \frac{B(2y)}{y} \E_{\nu_B}\Big[\frac{1}{\tau_{u^-}}{\bf 1}_{\{\xi_{u^-} \leq 2y,\;\xi_u \geq y\}}\Big]
\end{equation}
where $\E_{\nu_B}[\cdot]$ denotes expectation when the initial condition $(\xi_\emptyset, \tau_\emptyset)$ has distribution $\nu_B$ and where we have set 
$\nu_B(y) = \int_{{\mathcal E}}\nu_B(y,dv')$
in \eqref{representation cle} for the marginal density of the invariant probability measure $\nu_B$ with respect to $y$.
\end{prop}
We exhibit below a class of division rates $B$ and growth rate kernels $\rho$ that guarantees the existence of such an invariant probability measure. 
\subsubsection*{Construction of a nonparametric estimator}
Inverting \eqref{representation cle} and applying an appropriate change of variables, we obtain
\begin{equation} \label{first rep}
B(y)=\frac{y}{2}\frac{\nu_B(y/2)}{\E_{\nu_B}\Big[\frac{1}{\tau_{u^{-}}}{\bf 1}_{\{\xi_{u^-} \leq y,\;\xi_u \geq y/2\}}\Big]},
\end{equation}
provided the denominator is positive. This formula has no easy interpretation: it is obtained by some clever manipulation of the equation $\nu_B = {\mathcal P}_B\nu_B$. A tentative interpretation in the simplest case with no variability (so that $\tau_u=\tau$ for some fixed $\tau>0$ and for every $u \in \mathcal U$ is proposed in Section \ref{num}. 
Representation \eqref{first rep} also
suggests an estimation procedure, replacing the marginal density $\nu_B(y/2)$ and the expectation in the denominator by their empirical counterparts. To that end,
pick a kernel  function
$$K: [0,\infty)\rightarrow \R,\;\;\int_{[0,\infty)}K(y)dy=1,$$ 
and set $K_h(y)=h^{-1}K\big(h^{-1}y\big)$ for $y\in [0,\infty)$ and $h>0$. Our estimator is defined as
\begin{align} 
\widehat B_n(y) 
& = \frac{y}{2}\,\frac{n^{-1}\sum_{u\in {\mathcal U}_n}K_h(\xi_u-y/2)}{n^{-1}\sum_{u \in {\mathcal U}_n} \frac{1}{\tau_{u^-}}{\bf 1}_{\displaystyle \{\xi_{u^-}\leq y, \xi_u \geq y/2\}} \bigvee \varpi},
\label{def estimator}
\end{align}
where $\varpi >0$ is a threshold that ensures that the estimator is well defined in all cases and $x \bigvee y = \max\{x,y\}$. Thus $(\widehat B_n(y), y\in {\mathcal D})$ is specified by the choice of the kernel  $K$, the bandwidth $h>0$ and the threshold $\varpi>0$. 
\begin{assumption} \label{prop K} The function $K$ is bounded with  compact support, and for some integer $n_0 \geq 1$, we have
$\int_{[0,\infty)}x^kK(x)dx={\bf 1}_{\{k=0\}}\;\;\text{for}\;\;0 \leq k\leq n_0.$
\end{assumption}
\subsection{Error estimates} \label{upper rates}
We assess the quality of $\widehat B_n$ in squared-loss error over compact intervals ${\mathcal D}$. We need to specify local smoothness properties of $B$ over ${\mathcal D}$, together with general properties that ensure that  the empirical measurements in \eqref{def estimator} converge with an appropriate speed of convergence. This amounts to impose an appropriate behaviour of $B$ near the origin and infinity.
\subsubsection*{Model constraints}
For $\lambda > 0$ such that $2^\lambda > \sup \mathcal E/\inf \mathcal E > 0$ and a vector of positive constants 
$\mathfrak{c}=(r,m, \ell, L)$,
introduce the class ${\mathcal F}^\lambda(\mathfrak{c})$  
of continuous functions
$B:[0,\infty)\rightarrow [0,\infty)$  such that 
\begin{equation} \label{loc control}
\int_{0}^{r/2}x^{-1}B(2x)dx \leq L,\;\;\;\int_{r/2}^{r}x^{-1}B(2x)dx \geq \ell,
\end{equation}
and
\begin{equation} \label{poly control}
B(x) \geq m\, x^\lambda\;\;\text{for}\;\;\;x\geq r.
\end{equation}
\begin{remark}
Similar conditions on the behaviour of $B$ can also be found in \cite{DG}, in a deterministic setting. 
\end{remark}
\begin{remark}
Assumption \ref{basic assumption} is satisfied as soon as $B \in {\mathcal F}^\lambda(\mathfrak{c})$ (and $\rho$ is defined on a compact ${\mathcal E} \subset (0,\infty)$ of course). 
\end{remark}
Let $\rho_{\min}, \rho_{\max}$ be two positive finite measures on ${\mathcal E}$ such that $\rho_{\max}-\rho_{\min}$ is a positive measure and $\rho_{\min}({\mathcal E})>0$.
We define ${\mathcal M}(\rho_{\min},\rho_{\max})$ as the class of Markov transitions $\rho(v,dv')$ on ${\mathcal E}$ such that
\begin{equation} \label{geo ergo variabilite}
\rho_{\min}(A) \leq \rho(v,A) \leq \rho_{\max}(A),\;\;A \subset {\mathcal E},v \in {\mathcal E}.
\end{equation}
\begin{remark}
Control \eqref{geo ergo variabilite} ensures the geometric ergodicity of the process of variability in the growth rate.
\end{remark}
Let us be given in the sequel a vector of positive constants $\mathfrak{c}=(r,m, \ell, L)$ and $0<e_{\min} \leq  e_{\max}$ such that ${\mathcal E} \subset [e_{\min}, e_{\max}]$. We introduce the Lyapunov function
\begin{equation} \label{first def Lyapu}
{\mathbb V}(x,v)={\mathbb V}(x)=\exp\big(\tfrac{m}{e_{\min} \lambda}x^\lambda\big)\;\;\text{for}\;\;(x,v)\in {\mathcal S}.
\end{equation}
The function ${\mathbb V}$ controls the rate of the geometric ergodicity of the chain with transition ${\mathcal P}_B$ and will appear in the proof of Proposition \ref{prop transition} below. Define
$$\delta = \delta(\mathfrak{c}):=\frac{1}{1-2^{-\lambda}} \exp\big(- (1-2^{-\lambda})\tfrac{m}{e_{\max} \lambda} r^\lambda\big)\rho_{\max}({\mathcal E})
.$$
\begin{assumption}[The sparse tree case]
\label{contrainte constante}
Let $\lambda >0$. We have $\delta(\mathfrak{c})<1$.
\end{assumption}
In the case of the full tree observation scheme, we will need more stringent (and technical) conditions on $\mathfrak{c}$. Let $\gamma_{B,{\mathbb V}}$ denote the spectral radius of the operator ${\mathcal P}_B - 1 \otimes \nu_B$ acting on the Banach space of functions $g:{\mathcal S}\rightarrow \R$ such that 
$$\sup\{|g(\boldsymbol{x})|/{\mathbb V}(\boldsymbol{x}), \boldsymbol{x}\in {\mathcal S}\}<\infty,$$
where ${\mathbb V}$ is defined in \eqref{first def Lyapu} above. 
\begin{assumption}[The full tree case] \label{the full tree assumption}
We have 
$\delta(\mathfrak{c})<\tfrac{1}{2}$
and moreover
\begin{equation} \label{condition norme operateur}
\sup_{B \in {\mathcal F}^\lambda(\mathfrak{c})}\gamma_{B,{\mathbb V}}<\tfrac{1}{2}.
\end{equation}
\end{assumption}
\begin{remark}
It is possible to obtain bounds on $\mathfrak{c}$ so that \eqref{condition norme operateur} holds, by using explicit (yet intricate) bounds on $\gamma_{B,{\mathbb V}}$ following Fort {\it et al.} or \cite{FMP}, Douc {\it et al.} \cite{DMR}, see also Baxendale \cite{B}.
\end{remark}
\subsubsection*{Rate of convergence}
We are ready to state our main result. For $s>0$, with $s=\lfloor s\rfloor + \{s\}$, $0< \{s\} \leq 1$ and $\lfloor s\rfloor$ an integer, introduce the H\"older space ${\mathcal H}^s({\mathcal D})$ of functions $f:{\mathcal D}\rightarrow \R$ possessing a derivative of order $\lfloor s \rfloor$ that satisfies
\begin{equation} \label{def sob}
|f^{\lfloor s \rfloor}(y)-f^{\lfloor s \rfloor}(x)| \leq c(f)|x-y|^{\{s\}}.
\end{equation}
The minimal constant $c(f)$ such that \eqref{def sob} holds defines a semi-norm $|f|_{{\mathcal H}^s(\mathcal{D})}$. We equip the space ${\mathcal H}^s(\mathcal D)$ with the norm 
$$\|f\|_{{\mathcal H}^s(\mathcal D)} = \|f\|_{L^\infty({\mathcal D})} + |f|_{{\mathcal H}^s({\mathcal D})}$$ and the H\"older balls
$${\mathcal H}^s({\mathcal D}, M) = \{B,\;\|B\|_{{\mathcal H}^s({\mathcal D})} \leq M\},\;M>0.$$
\begin{theorem} \label{upper bound}
Work under Assumption \ref{contrainte constante} in the sparse tree case and Assumption \ref{the full tree assumption} in the full tree case.
Specify $\widehat B_n$ with a kernel $K$ satisfying Assumption \ref{prop K} for some $n_0>0$ and
$$h_n=c_0n^{-1/(2s+1)},\;\; \varpi_n \rightarrow 0.$$
For every $M>0$,  there exist $c_0=c_0(\mathfrak{c}, M)$ and $d(\mathfrak{c})\geq0$ such that for every $0<s<n_0$ and every compact interval ${\mathcal D}\subset (d(\mathfrak{c}),\infty)$ such that $\inf {\mathcal D} \geq r/2$, we have
$$
\sup_{\rho, B}\E_{\mu}\big[\|\widehat B_n-B\|_{L^2({\mathcal D})}^2\big]^{1/2} \lesssim \varpi_n^{-1}(\log n)^{1/2}  n^{-s/(2s+1)},$$
where the supremum is taken over
$$\rho \in  {\mathcal M}(\rho_{\min},\rho_{\max})\;\;\text{and}\;\;B \in {\mathcal F}^\lambda(\mathfrak{c}) \cap {\mathcal H}^s({\mathcal D}, M),$$
and $\E_{\mu}[\cdot]$ denotes expectation with respect to any initial distribution $\mu(d\boldsymbol{x})$ for $(\xi_\emptyset,\tau_\emptyset)$ on ${\mathcal S}$ such that $\int_{\mathcal S}{\mathbb V}(\boldsymbol{x})^2\mu(d\boldsymbol{x})<\infty$.
 \end{theorem}
Several remarks are in order: {\bf 1)} Since $\varpi_n$ is arbitrary, we obtain the classical rate $n^{-s/(2s+1)}$ (up to a log term) which is optimal in a minimax sense for density estimation. It is presumably optimal in our context, using for instance classical techniques for nonparametric estimation lower bounds on functions of transition densities of Markov chains, see for instance \cite{GHR}. {\bf 2)} The extra logarithmic term is due to technical reasons: we need it in order to control the decay of correlations of the observations over the full tree structure. {\bf 3)} The knowledge of the smoothness $s$ that is needed for the construction of $\widehat B_n$ is not realistic in practice. An adaptive estimator could be obtained by using a data-driven bandwidth in the estimation of the invariant density $\nu_B(y/2)$ in \eqref{def estimator}. The Goldenschluger-Lepski bandwidth selection method \cite{GL2}, see also \cite{DHRR} would presumably yield adaptation, but checking the assumptions still requires a proof in our setting. We implement data-driven bandwidth in the numerical Section \ref{numerical implementation} below. 
\section{Numerical implementation} \label{numerical implementation}

\subsection{Protocol and results}
\subsubsection*{Generating simulated data}
Given a division rate $B(x)$, a growth rate kernel $\rho$, an initial distribution $\mu({d{\boldsymbol{x}}})$ for the node $(\xi_{\emptyset},\tau_{\emptyset})$ (as in Theorem \ref{upper bound}) and a dataset size $n=2^{N_n}$, we simulate the full tree and the sparse tree schemes 
recursively:
\begin{enumerate}
\item Given $(\xi_{u^-},\tau_{u^-}),$ we select at random its lifetime  $\zeta_{u^-}$ (by a rejection sampling algorithm) with probability density
$$t \leadsto B\big(\xi_{u^-}\exp(\tau_{u^-} t)\big)\exp\Big(-\int_0^tB\big(\xi_{u^-} \exp(\tau_{u^-} s)\big)ds\Big).
$$
following the computations of Section~\ref{estimation of the division rate}.
\item
We derive the size at birth $\xi_u$ for the two offsprings (with $u = (u^-, 0)$ and $(u^-,1)$ with obvious notation) by Formula~\eqref{fundamental}.
\item We simulate at random the growth rates $\tau_{u}$ (by the rejection sampling algorithm) according to the distribution $\rho(\tau_{u^-}, dv).$ 
\item For the sparse tree case, we select only one offspring (either $(u^-,0)$ of $(u^-,1)$), whereas we keep both for the full tree case.
\end{enumerate}
In order to stay in line with previous simulations of \cite{DHRR} we pick $B(x)=x^2$. We fix  $\mu(d{\boldsymbol x})$ as the uniform distribution over $[1/3,3]\times {\mathcal E}$, with ${\mathcal E} = [0.2, 3]$.
As for the growth rate kernel,  we implement 
$$\rho(v,dv')= g(v'-v)dv'$$
where $g$ is a uniform distribution over $[1-\alpha, 1+\alpha]$ for some $\alpha >0$, and dilated by a scaling factor so that $\big(\int (v'-v)^2\rho(v,dv')\big)^{1/2}=1/2$. We also condition the values of $\tau_u$ to stay in ${\mathcal E}$ (by rejection sampling).
\subsubsection*{Implementing  $\widehat B_n$}
We implement $\widehat{B}_n$ using Formula~\eqref{def estimator}. We pick a standard Gaussian kernel $K(x)=(2\pi)^{-1/2}\exp(-x^2/2)$, for which $n_0=1$ in Assumption~\eqref{prop K}; henceforth we expect a rate of convergence of order $n^{-1/3}$ at best.
We evaluate $\widehat{B}_n$ on a regular grid $x_1=\Delta x,\cdots x_m, =m \Delta x$ with $\Delta x = n^{-1/2}$ and $x_m=5$. Thus $x_m$ is large enough so that $\nu_B(x/2)$ becomes negligible for $x\geq x_m$ and $\Delta x$ is smaller than $n^{-1/3}$ to avoid numerical discrepancies. 
For tractability purposes, we wish to avoid the use of any relationship between the nodes $u \in {\mathcal U}_n$. Indeed, whereas it is quite easy to label $u^-$ and $u$ in the sparse tree case, it is a bit more difficult to track the parent of each individual in the full tree case if we do not want to double the memory. As a consequence, we simply reformulate~\eqref{def estimator} into
\begin{align} 
\widehat B_n(y) 
& = \frac{y}{2}\,\frac{n^{-1}\sum_{u\in {\mathcal U}_n}K_h(\xi_u-y/2)}{n^{-1}\sum_{u \in {\mathcal U}_n} \frac{1}{\tau_{u}}{\bf 1}_{\displaystyle \{\xi_{u}\leq y \leq \xi_{u}e^{\tau_{u}\zeta_{u}} \}} \bigvee \varpi}.
\label{def estimator2}
\end{align}
We take $h_n=n^{-1/3}$ for the bandwidth according to Theorem~\ref{upper bound} to serve as a proof of concept. Data-driven choices could of course be made, such as  the Goldenschluger and Lepski's method~\cite{GL2,DHRR}, and improve the already fairly good results shown in Figure~\ref{fig:nvariab2}.
Finally, we also test whether taking into account variability in the growth rate improves significantly or not the estimate of $B,$ replacing $\tau_u$ by its mean value $n^{-1}\sum_{u\in {\cal U}_n} \tau_u$ everywhere in Formula~\eqref{def estimator2}, thus ignoring growth variability in that case. 
\subsubsection*{Numerical results}
We display our numerical results as specified above in Figures~\ref{fig:nvariab}, \ref{fig:nvariab2} and \ref{fig:n2puis10}. Figure~\ref{fig:nvariab} displays the reconstruction of $B$ on the full tree scheme for a simulated sample of size $n=2^{17}$. At a visual level, we see that the estimation deteriorates dramatically when the variability is ignored in the region where $\nu_B$ is small, while our estimator \eqref{def estimator2} still shows good performances. 
\begin{figure}[ht]
\begin{minipage}{16cm}
\includegraphics[width=11cm, height=7cm]{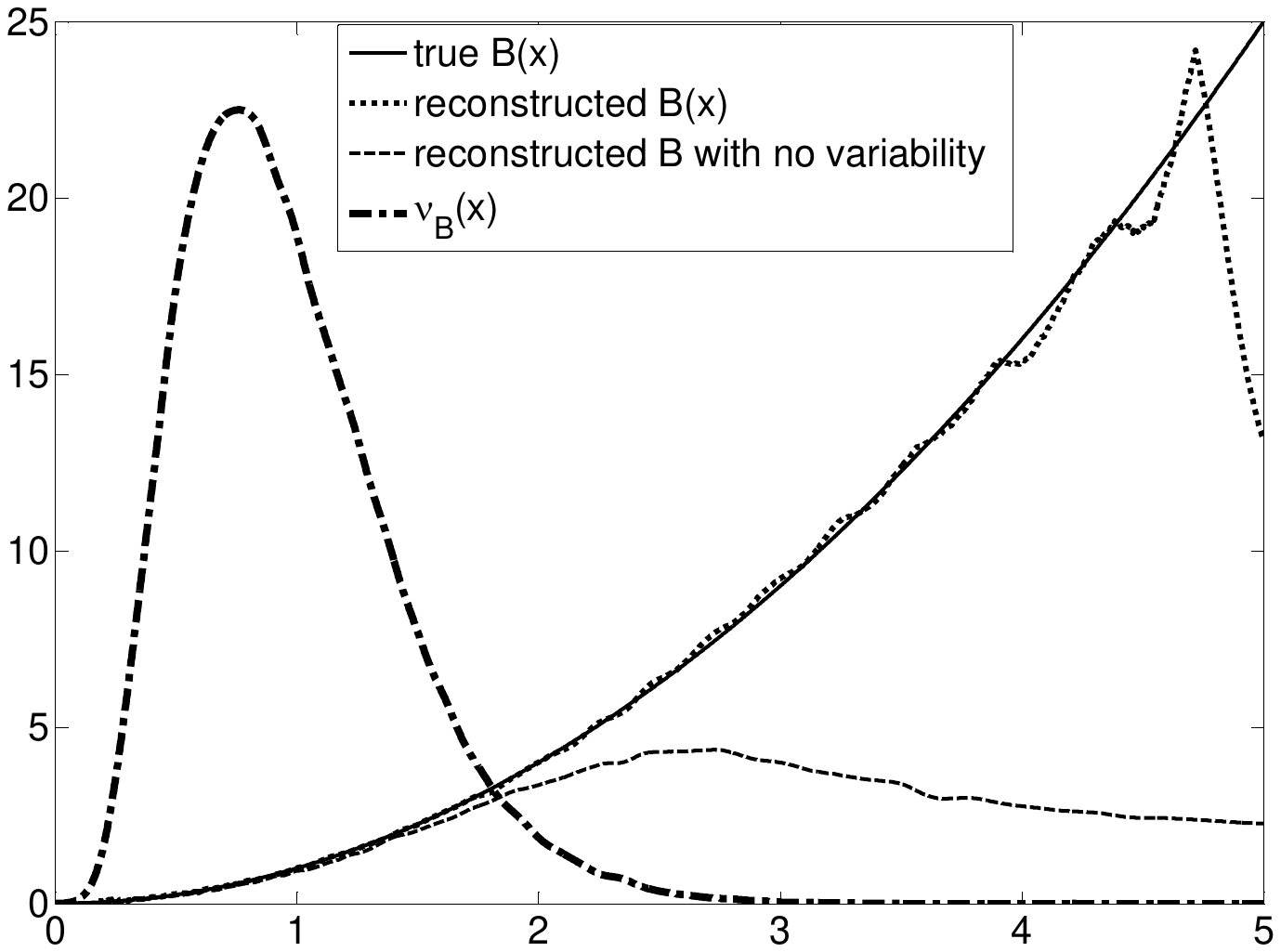}
\end{minipage} 
\caption{\label{fig:nvariab} {\it Reconstruction for $n=2^{17}.$  When the variability in the growth rate is ignored, the estimate reveals unsatisfactory. The parameter values are the reference ones.}}
\end{figure}
\begin{figure}[ht]
\begin{minipage}{16cm}
\includegraphics[width=11cm, height=7cm]{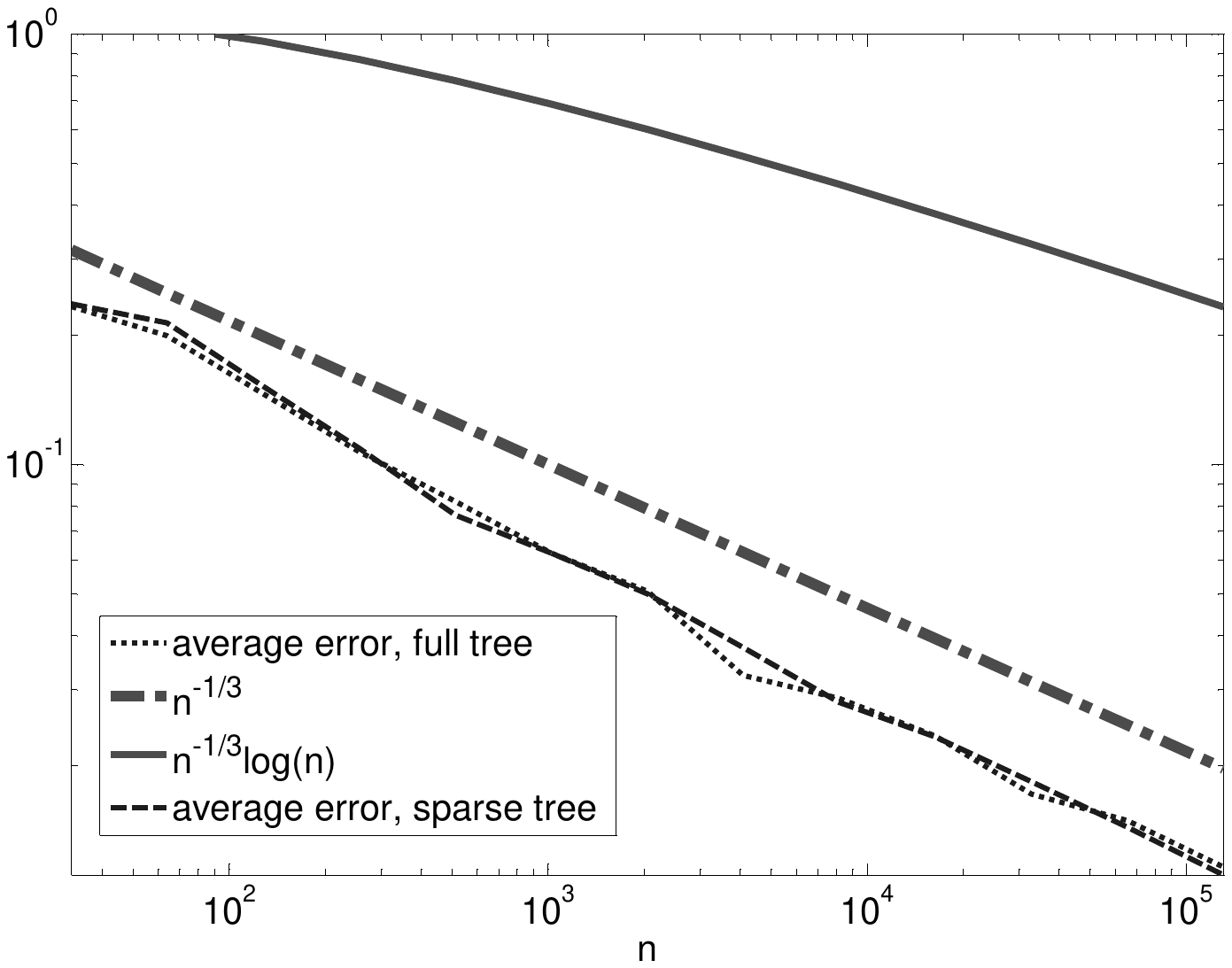}
\end{minipage} 
\caption{\label{fig:nvariab2} {\it Error v.s. $n$ for the full tree and the sparse tree case on a log-log scale. The error actually proves better than the upper rate of convergence announced in Theorem~\ref{upper bound}, and $\varpi$ may be taken smaller than $\log(n)$. Estimates are comparable for both schemes. The parameter values are the reference ones.}}
\end{figure}

In Figure \ref{fig:nvariab2}, we plot on a log-log scale the empirical mean error of our estimation procedure for both full tree and sparse tree schemes. The numerical results agree with the theory. The empirical error is computed as follows: we compute
\begin{equation} \label{error def}
e_i = \frac{\|\widehat B - B\|_{\Delta x, m}}{\|B\|_{\Delta x, m, \varpi}},\;\;i=1,\ldots, M,
\end{equation}
where $\|\cdot\|_{\Delta x,m, \varpi}$ denotes the discrete norm over the numerical sampling described above, conditioned on the fact that the denominator in \eqref{def estimator2} is larger than $\varpi/\log(n)$. We end up with a mean-empirical error 
$\overline{e}=M^{-1}\sum_{i=1}^Me_i$.
The number of Monte-Carlo samples is chosen as $M=100$. In Figure~\ref{fig:n2puis10}, we explore further the degradation of the estimation process on the region where $\nu_B$ is small, plotting $95\%$ confidence intervals of the empirical distribution of the estimates, based on $M=100$ Monte-Carlo samples. Finally, Table \ref{a table!} displays the relative error for the reconstruction of $B$ according to \eqref{error def}. The standard deviation is computed as $(M^{-1}\sum_{i=1}^M\big(e_i-\overline{e}\big)^2)^{1/2}$.
\begin{table} \label{a table!}
\begin{center}
\begin{tabular}{ccccccc}\label{tab:num}
$\log_2(n)$  & $5$ & $6$ & $7$ & $8$ & $9$ & $10$
\\
\hline
$\overline{e}$  & 0.2927 & 0.1904 & 0.1460 & 0.1024 & 0.0835 & 0.0614
\\
\hline
std. dev. & 0.1775 & 0.0893 & 0.0627 & 0.0417 & 0.0364 & 0.0241 \\
\hline
\end{tabular}
\end{center}
\caption{{\it Relative error $\overline{e}$ for $B$ and its standard deviation, with respect to $n$ (on a log scale). The error is computed using \eqref{error def} with $\varpi=1/\log(n).$}}
\end{table}
\begin{figure}[ht]
\begin{minipage}{16cm}
\includegraphics[width=12cm, height=7cm]{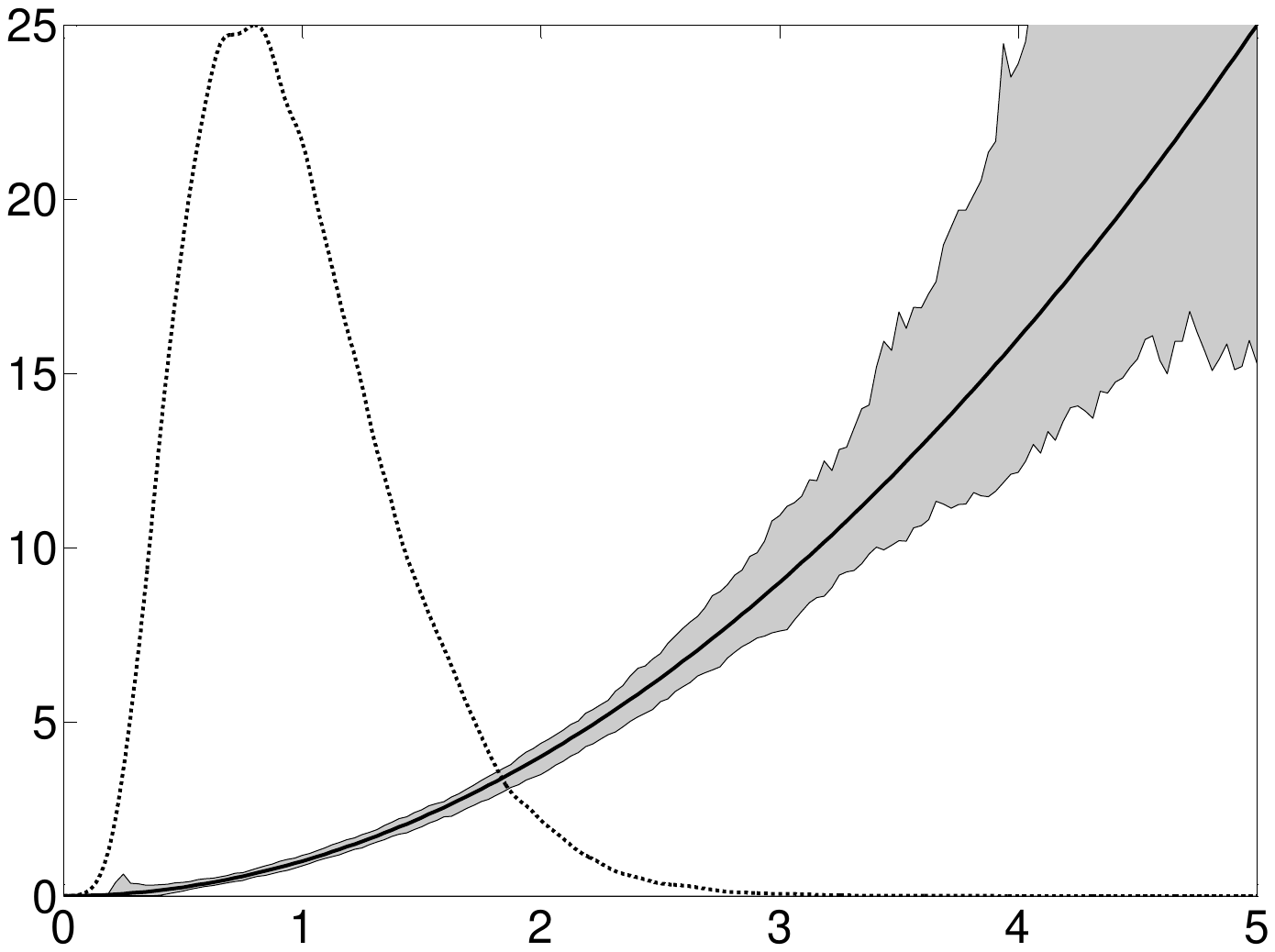}
\end{minipage} 
\caption{\label{fig:n2puis10} {\it Reconstruction for $n=2^{10},$ error band for $95\%$, full tree case, over $M=100$ simulations, with $\varpi=1/{n}$ in order to emphasise that the larger $x,$ the smaller $\nu_B$ and the larger the error estimate}.}
\end{figure}
We also carried out control experiments for other choices of variability kernel $\rho(v,dv')$ for the growth rate. These include $\rho(v,dv')=g(v')dv'$, so that the variability of an individual is not inherited from its parent, a Gaussian density for $g$ with the same prescription for the mean and the variance as in the uniform case, conditioned to live on $[e_{\min}, e_{\max}]$. We also tested the  absence of variability, with $\rho(v,dv')=\delta_{\tau}(dv')$, with $\tau=1$. None of these control experiments show any significant difference from the case displayed in Figures~\ref{fig:nvariab}, \ref{fig:nvariab2} and \ref{fig:n2puis10}. 
 
\subsubsection*{Analysis on {\it E. coli} data}

Finally, we analyse a dataset obtained through microscopic time-lapse imaging of single bacterial cells growing in rich medium, by Wang, Robert et al. \cite{Robert} and by Stewart et al. \cite{Stewart}. Thanks to a microfluidic set-up, the experimental conditions are well controlled and stable, so that the cells are in a steady state of growth (so-called balanced growth). The observation scheme corresponds to the sparse tree case for the data from Wang, Robert et al.: at each generation, only one offspring is followed. On the contrary, data corresponds to the full tree case for the data by Stewart et al., where the cells grow in a culture. The growth and division of the cells is followed by microscopy, and image analysis allows to determine the time evolution of the size of each cell, from birth to division. We picked up the quantities of interest for our implementation:  for each cell, its size at birth, growth rate and lifetime. 
We consider that cells divide equally into two daughter cells, neglecting the small differences of size at birth between daughter cells.
 Each cell grows exponentially fast, but  growth rates exhibit variability. 

Our data is formed by the concatenation of several  lineages, each of them composed with a line of offsprings coming from a first single cell picked at random in a culture. Some of the first and last generations were not considered in order to avoid any experimental disturbance linked either to  non stationary conditions or to aging of the cells. 

We proceed as in the above protocol. Figure~\ref{fig:exper} shows the reconstructed $B$ and $\nu_B$ for a sample of  $n=2335$ cells for the sparse tree data, $n=748$ for the full tree data. Though much more precise and reliable, thanks both to the experimental device and the reconstruction method, our results are qualitatively in accordance with previous indirect reconstructions carried out in \cite{DMZ} on old datasets published in \cite{Kubitschek} back in 1969.

The reconstruction of the division rate is prominent here since it appears to be the last component needed for a full calibration of the model. Thus, our method provides  biologists with a complete understanding of the size dependence of the biological system. 
Phenotypic variability between genetically identical cells has recently received growing attention with the recognition that it can be genetically controlled and subject to selection pressures \cite{Kaern}. Our mathematical framework allows the incorporation of this variability at the level of individual growth rates. It should allow the study of the impact of variability on the population fitness and should be of particular importance to describe the growth of populations of cells exhibiting high variability of growth rates. Several examples of high variability have been described, both in genetically engineered or natural bacterial populations \cite{Sturm,Tan_Marguet_You_2009}.

\begin{figure}[ht]
\begin{minipage}{16cm}
\includegraphics[width=12cm, height=7cm]{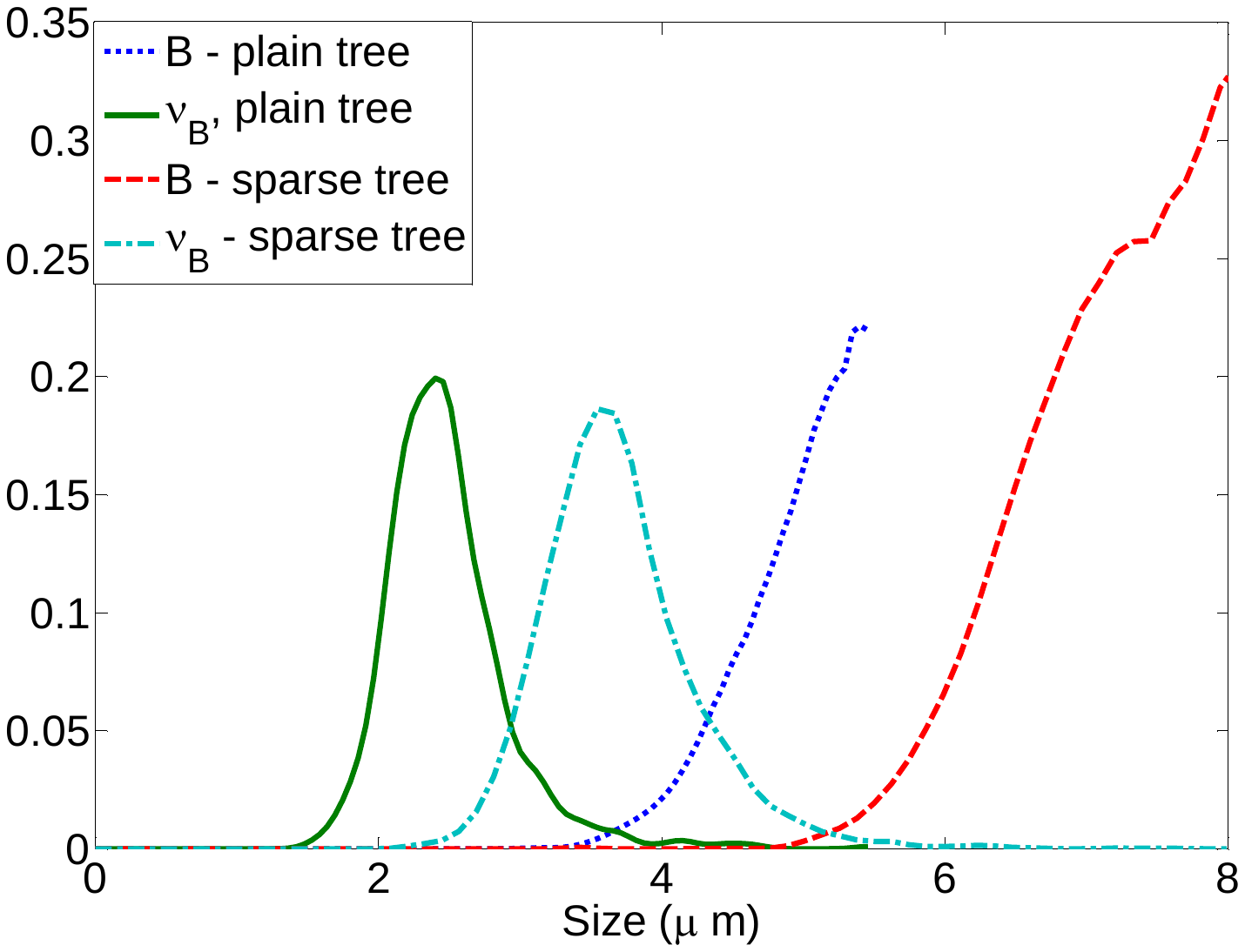}
\end{minipage} 
\caption{\label{fig:exper} {\it Estimation of $B$ (dotted line and dashed line resp.) and $\nu_B$ (solid line and dash-dotted line resp.) on experimental data of E. coli dividing cells, for resp. a sparse tree and a full tree experiment. $n=2335$ for the sparse tree, $n=748$ for the full tree. The experimental conditions being different (temperature and nutrient), these results are not supposed to be identical, yet a generic pattern appears for $B$, that could serve as a basis for future biological studies.}}
\end{figure}

\subsection{Link with the deterministic viewpoint} \label{num}
Considering the reconstruction formula~\eqref{def estimator}, let us give here some insight from a deterministic analysis perspective.  For the sake of clarity, let us focus on the simpler case when there is no variability, so that for all $u\in {\cal U}_n$ we have $\tau_u=\tau>0$ a fixed constant. Formula~\eqref{def estimator} comes from~\eqref{first rep}, which in the case $\tau_u=\tau $ simplifies further into
\begin{equation}\label{formule B}
B(y)=\frac{\tau y}{2}\frac{\nu_B(y/2)}{\E_{\nu_B}\Big[{\bf 1}_{\{\xi_{u^-} \leq y,\;\xi_u \geq y/2\}}\Big]}=\frac{\tau y}{2}\frac{\nu_B(y/2)}{\int_{y/2}^y \nu_B(z)dz}.
\end{equation}
We also notice  that, in this particular case, 
we do not need to measure the lifetime of each cell in order to implement \eqref{formule B}.
Define $N(y)=\tfrac{1}{2}\frac{\nu_B(y/2)}{B(y)},$ or equivalently $\nu_B(x)=2B(2x)N(2x).$ Differentiating \eqref{formule B}, we obtain 
$$
\partial_x(\tau x N)=2B(2x)N(2x)-B(x)N(x)
$$
which corresponds to the stationary state linked to the equation
\begin{equation}\left\{\begin{array}{l}
{\partial_t}n(t,x)+ \tau \partial_x \big(x\, n(t,x)\big)=2B(2x)n(t,2x)-B(x)n(t,x),
 \\ \\
n(0,x)= n^{(0)}(x), x \geq 0.
\end{array}
\right.\label{eq:transport-fragmentation conservative}
\end{equation}
Eq. \eqref{eq:transport-fragmentation conservative} exactly corresponds to the behaviour of the tagged cell of Section \ref{a many-to-one formula via a tagged branch} below, in a (weak) sense:
$$n(t,dx) = \PP(\chi(t)\in dx)$$
where $\chi(t)$ denotes the size at time $t$ along a branch picked at random, see Section \ref{a many-to-one formula via a tagged branch}. Existence and uniqueness of an invariant measure $\nu_B$ has an analogy to the existence of a steady state solution for the  PDE \eqref{eq:transport-fragmentation conservative}, and the convergence of the empirical measure to the invariant rejoins the stability of the steady state \cite{PPS3}. The equality $\nu_B(x)=2B(2x)N(2x)$ may be interpreted as follows: $N(x)$ is the steady solution of Eq.~\eqref{eq:transport-fragmentation conservative}, and represents the probability density of a cell population dividing at a rate $B$ and growing at a rate $x \tau$, but when only one offspring remains alive at each division so that the total quantity of cells remains constant. The fraction of dividing cells is represented by the term $B(x)N(x)$ in the equation, with distribution given by $\tfrac{1}{2}\nu_B(x/2),$ whereas the fraction of newborn cells is $2B(2x)N(2x)$. Eq.~\eqref{formule B} can be written in terms of $BN$ as
\begin{equation} \label{surprise}
B(y)=\frac{\tau y BN(y)}{\int_y^{2y} B(z)N(z) dz}.
\end{equation}
This also highlights why we obtain a rate of convergence of order $n^{-s/(2s+1)}$ rather than the rate $n^{-s/(2s+3)}$ obtained with indirect measurements as in \cite{DHRR}. In that latter case, we observe a $n$-sample with distribution $N$. As shown in \cite{DHRR}, one differentiation is necessary  to estimate $B$ therefore we have a degree of ill-posedness of order 1. In the setting of the present paper, we rather observe a sample with distribution $BN$,  and $B$ can be recovered directly from \eqref{surprise} and we have here a degree of ill-posedness of order $0$.
\section{Proofs} \label{proofs}
We first prove Theorem \ref{sol transport general} in Sections \ref{a many-to-one formula via a tagged branch} and \ref{proof th1}. The strategy consists in obtaining a so-called many-to-one formula (Proposition \ref{many-to-one}) that enables to relate additive functionals of the whole Markov tree to a special one-dimensional process that consists of following at random a branch on the tree. It suffices to check in Section \ref{proof th1} that this randomly tagged branch integrated against appropriate test functions satisfies the desired transport-fragmentation equation.  Section \ref{sec geo ergo} studies at length the Markov transition ${\mathcal P}_B$. We first prove the key representation formula for $B$ obtained in Proposition \ref{rep B via inv}. We then quantify the geometric ergodicity of the model by a standard Lyapunov technique (Proposition \ref{prop transition}). In Section \ref{covariance inequalities} we subsequently apply the geometric ergodicity of the transition ${\mathcal P}_B$ by establishing covariance inequalities on a tree in Propositions \ref{moment bound} and \ref{moment bound bis}; these are the crucial tools to later control the convergence rate of the estimator. We need in particular to study the covariance of delta-like functions with supremum norm increasing to infinity with our asymptotic, and this explains the relative technical length of our estimates. This enables to further control in Section \ref{rate of convergence for the empirical measure} a rate of convergence for the empirical measure in Propositions \ref{moment lemma} and \ref{moment lemma bis}. The fact that we work on a tree with an non-reversible Markov transition and delta-like test functions is an extra technical difficulty. Finally, we can prove Theorem \ref{upper bound} in Section \ref{final proof} for the rate of convergence of our estimator with a classical trade-off technique between a bias and a variance term, thanks to the tools developed in the preceding sections and in particular in Section \ref{further invariant} where some useful estimates for the invariant measure are established.

The notation $\lesssim$ means inequality up to a constant that does not depend on $n$. We set $a_n \sim b_n$ when $a_n \lesssim b_n$ and $b_n \lesssim a_n$ simultaneously. A mapping $f:{\mathcal E}\rightarrow \R$ or $g:[0,\infty)\rightarrow \R$ is implicitly identified as a function on ${\mathcal S}$ via $f(x,v)=f(x)$ and $g(x,v)=g(v)$.

\subsection{A many-to-one formula via a tagged cell} \label{a many-to-one formula via a tagged branch}
For $u \in {\mathcal U}$, we set $m^i u$ for the $i$-th parent along the genealogy of $u$. Define
$$\overline{\tau_t^u} = \sum_{i = 1}^{|u|}\tau_{m^i u}\zeta_{m^i u}+\tau_t^u(t-b_u)\;\;\text{for}\;\;t\in [b_u, b_u + \zeta_u)$$
and $0$ otherwise for the cumulated growth rate along its ancestors up to time $t$.  In the same spirit as tagged fragments in fragmentation processes (see the book by Bertoin \cite{bertoin} for instance) 
we pick a branch at random along the genealogical tree at random: for every $k \geq 1$, if $\vartheta_k$ denotes the node of the tagged cell at the $k$-th generation, we have
$$\PP(\vartheta_k = u)=2^{-k}\;\;\text{for every}\;\;u\in {\mathcal U}\;\;\text{such that}\;\;|u|=k,$$
and $0$ otherwise.
For $t\geq 0$,  the relationship
$$b_{\vartheta_{C_t}} \leq t < b_{\vartheta_{C_t}}+\zeta_{\vartheta_{C_t}}$$ 
uniquely defines a counting process $(C_t, t \geq 0)$ with $C_0=0$. The process $C_t$ enables in turn to define a tagged process 
of size, growth rate and cumulated growth rate via
$$
\big(\chi(t), {\mathcal V}(t), \overline{{\mathcal V}}(t)\big) = \Big(\xi_t^{\vartheta_{C_t}}, \tau_{t}^{\vartheta_{C_t}}, \overline{\tau_t^{\vartheta_{C_t}}}\Big)\;\;\text{for}\;\;t \in [b_{\vartheta_{C_t}},b_{\vartheta_{C_t}}+\zeta_{\vartheta_{C_t}})
$$
and $0$ otherwise. We have the representation
\begin{equation} \label{rep chi}
\chi(t) = \frac{xe^{\overline{{\mathcal V}}(t)}}{2^{C_t}}
\end{equation}
and since ${\mathcal V}(t) \in [e_{\min},e_{\max}]$, we note that
\begin{equation} \label{controle croissance cumulee}
e_{\min}t \leq \overline{{\mathcal V}}(t) \leq e_{\max}t.
\end{equation}
The behaviour of $\big(\chi(t), {\mathcal V}(t), \overline{{\mathcal V}}(t)\big)$ can be related to certain functionals of the whole particle system via a so-called many-to-one formula. 
This is the key tool to obtain Theorem  \ref{sol transport general}.
\begin{prop}[A many-to-one formula] \label{many-to-one}
Work under Assumption \ref{basic assumption}. For $x \in (0,\infty)$, let $\PP_x$ be defined as in Lemma \ref{counting property}. For every $t\geq 0$, we have
$$\E_x\big[\phi\big(\chi(t), {\mathcal V}(t), \overline{{\mathcal V}}(t)\big)\big] = \E_x\Big[\sum_{u \in {\mathcal U}}\xi_t^u\frac{e^{-\overline{\tau_t^u}}}{x}\phi\big(\xi_t^u, \tau_t^u, \overline{\tau_t^u}\big)\Big]$$
for every $\phi:{\mathcal S}\times [0,\infty) \rightarrow [0,\infty)$.
\end{prop} 
\begin{proof}[Proof of Proposition \ref{many-to-one}]
For $v \in {\mathcal U}$, set $I_v = [b_{v}, b_v+\zeta_v)$. By representation \eqref{rep chi}, we have
\begin{align*}
\E_{x}\big[\phi\big(\chi(t), {\mathcal V}(t), \overline{{\mathcal V}}(t)\big)\big] & = \E_{x}\big[\phi\big(\tfrac{xe^{\overline{{\mathcal V}}(t)}}{2^{C_t}}, {\mathcal V}(t), \overline{{\mathcal V}}(t)\big)\big] \\
& =  \E_{x}\Big[\sum_{v \in {\mathcal U}}\phi\big(\tfrac{xe^{\overline{\tau_t^v}}}{2^{|v|}}, \tau_t^v, \overline{\tau_t^v}\big){\bf 1}_{\{t\in I_v, v=\vartheta_{C_t}\}}\Big].
\end{align*}
Introduce the discrete filtration ${\mathcal H}_n$ generated by $(\xi_u, \zeta_u, \tau_u)$ for every $u$ such that $|u|\leq n$. Conditioning with respect to ${\mathcal H}_{|v|}$ and noting that on $\{t \in I_v\}$, we have
$$\PP\big(\vartheta_{C_t} = v\,|\,{\mathcal H}_{|v|}\big)=\frac{1}{2^{|v|}} = \frac{\xi_v e^{-\overline{\tau_{b_v}^v}}}{x},$$
we derive
\begin{align*}
\E_{x}\Big[\sum_{v \in {\mathcal U}}\phi\big(\tfrac{xe^{\overline{\tau_t^v}}}{2^{|v|}}, \tau_t^v, \overline{\tau_t^v}\big){\bf 1}_{\{t\in I_v, v=\vartheta_{C_t}\}}\Big]  = \,&\E_{x}\Big[\sum_{v \in {\mathcal U}}\xi_v \frac{e^{-\overline{\tau_{b_v}^v}}}{x}\phi\big(\tfrac{xe^{\overline{\tau_t^v}}}{2^{|v|}}, \tau_t^v, \overline{\tau_t^v}\big){\bf 1}_{\{t\in I_v\}}\Big] \\
=\,& \E_x\Big[\sum_{u \in {\mathcal U}}\xi_t^u\frac{e^{-\overline{\tau_t^u}}}{x}\phi\big(\xi_t^u, \tau_t^u, \overline{\tau_t^u}\big)\Big].
\end{align*}
\end{proof}
\subsection{Proof of Theorem \ref{sol transport general}} \label{proof th1}
We fix $x \in (0,\infty)$ and prove the result for an initial measure $\mu_x$ as in Proposition \ref{many-to-one}. Let $\varphi \in {\mathcal C}^1_0({\mathcal S})$ be nonnegative. By \eqref{identity point measures} we have
$$
\langle n(t,\cdot), \varphi \rangle  = \E_x\big[\sum_{i=1}^\infty\varphi\big(X_i(t),Z_i(t)\big)\big] 
 = \E_x\big[\sum_{u \in {\mathcal U}}\varphi(\xi_t^u,\tau_t^u)\big] 
$$
and applying Proposition \ref{many-to-one}, we derive 
\begin{equation} \label{n via many-to-one}
\langle n(t,\cdot), \varphi \rangle =x\, \E_x\Big[\varphi\big(\chi(t), {\mathcal V}(t)\big)\frac{e^{\overline{{\mathcal V}}(t)}}{\chi(t)}\Big].
\end{equation}
For $h>0$, introduce the difference operator 
$$\Delta_h f(t) = h^{-1}\big(f(t+h)-f(t)\big).$$
We plan to study the convergence of $ \Delta_h\langle n(t,\cdot),\varphi \rangle$ as $h\rightarrow 0$ using representation \eqref{n via many-to-one} in restriction to the events $\{C_{t+h}-C_t=i\}$, for $i=0,1$ and $\{C_{t+h}-C_t \geq 2\}$.
Denote by ${\mathcal F}_t$ the filtration generated by the tagged cell $\big(\chi(s), {\mathcal V}(s),s \leq t\big)$. The following standard estimate proved in Appendix \ref{A1} will be later useful.
\begin{lemma} \label{counting property}
Assume that $B$ is continuous.  Let $x\in (0,\infty)$ and let $\mu_x$ be a probability measure on ${\mathcal S}$ such that $\mu_x(\{x\}\times {\mathcal E})=1$. Abbreviate $\PP_{\mu_x}$ by $\PP_x$. For small $h>0$, we have
$$\PP_x(C_{t+h}-C_t = 1\,|\,{\mathcal F}_t)=B\big(\chi(t)\big)h+h\,\varepsilon(h),$$
with the property $|\varepsilon(h)| \leq \epsilon(h) \rightarrow 0$ as $h \rightarrow 0$, for some deterministic $\epsilon(h)$, and
$$\PP_x(C_{t+h}-C_t \geq 2) \lesssim h^2.$$
\end{lemma}
Since $\varphi \in {\mathcal C}^1_0({\mathcal S})$, there exists $c(\varphi)>0$ such that $\varphi(y,v)=0$ if $y \geq c(\varphi)$. By \eqref{controle croissance cumulee}, we infer 
\begin{equation} \label{integrabilite reste}
\Big|\varphi\big(\chi(t), {\mathcal V}(t)\big)\frac{e^{\overline{{\mathcal V}}(t)}}{\chi(t)}\Big| \leq \sup_{y,v}\varphi(y,v)\frac{\exp(e_{\max}t)}{c(\varphi)}
\end{equation}
By Lemma \ref{counting property} and \eqref{integrabilite reste}, we derive
\begin{equation} \label{first estimate poisson}
\E_x\Big[\Delta_h\Big(\varphi\big(\chi(t), {\mathcal V}(t)\big)\frac{e^{\overline{{\mathcal V}}(t)}}{\chi(t)}\Big){\bf 1}_{\{C_{t+h}-C_t \geq 2\}}\Big] \lesssim h.
\end{equation}
On the event $\{C_{t+h}-C_t=0\}$, the process ${\mathcal V}(s)$ is constant for $s \in [t,t+h)$ and so is $\frac{e^{\overline{{\mathcal V}}(s)}}{\chi(s)}$ thanks to \eqref{rep chi}. It follows that
$$\Delta_h\Big(\varphi\big(\chi(t), {\mathcal V}(t)\big)\frac{e^{\overline{{\mathcal V}}(t)}}{\chi(t)}\Big)= \Delta_h\varphi\big(\chi(t), {\mathcal V}(s)\big)_{\big|_{s=t}}\frac{e^{\overline{{\mathcal V}}(t)}}{\chi(t)}$$
on $\{C_{t+h}-C_t=0\} $ and also
$$
\Big|\Delta_h\varphi\big(\chi(t), {\mathcal V}(s)\big)_{\big|_{s=t}}\frac{e^{\overline{{\mathcal V}}(t)}}{\chi(t)}\Big| \leq \sup_{y,v}|\partial_y\varphi(y,v)|xe_{\max}  \frac{\exp(2e_{\max}t)}{c(\varphi)}
$$
on $\{C_{t+h}-C_t=0\}$ likewise. Since $\PP_x(C_{t+h}-C_t=0)\rightarrow 1$ as $h\rightarrow 0$, by dominated convergence
\begin{align}
& \;x\,\E_x\Big[\Delta_h\Big(\varphi\big(\chi(t), {\mathcal V}(t)\big)\frac{e^{\overline{{\mathcal V}}(t)}}{\chi(t)}\Big){\bf 1}_{\{C_{t+h}-C_t =0\}}\Big] \nonumber \\
\rightarrow &\; x\,\E_x\big[\partial_1\varphi\big(\chi(t), {\mathcal V}(t)\big){\mathcal V}(t)e^{\overline{{\mathcal V}}(t)}\big]\;\;\text{as}\;\;h\rightarrow 0.
\label{convergence poisson 2}
\end{align}
By Proposition \ref{many-to-one} again, this last quantity is equal to $\langle n(t,dx,dv), xv\,\partial_x \varphi\rangle$. On $\{C_{t+h}-C_t=1\}$, we successively have
$$\chi(t+h) = \frac{1}{2}\chi(t)+\varepsilon_1(h),$$
$$\varphi\big(\chi(t+h), {\mathcal V}(t+h)\big) = \varphi\big(\chi(t)/2, {\mathcal V}(t+h)\big)+\varepsilon_2(h)$$
and
$$\exp\big(\overline{{\mathcal V}}(t+h)\big) = \exp\big(\overline{{\mathcal V}}(t)\big)+\varepsilon_3(h)$$
with the property $|\varepsilon_i(h)|\leq \epsilon_1(h) \rightarrow 0$ as $h\rightarrow 0$, where $\epsilon_1(h)$ is deterministic, thanks to \eqref{rep chi} and \eqref{controle croissance cumulee}. Moreover, 
$${\mathcal V}(t+h) = \tau_{\vartheta_{C_t+1}}\;\;\text{on}\;\;\{C_{t+h}-C_t=1\}.$$
It follows that
\begin{align*}
& \E_x\Big[\varphi\big(\chi(t+h), {\mathcal V}(t+h)\big)\frac{e^{\overline{{\mathcal V}}(t+h)}}{\chi(t+h)}{\bf 1}_{\{C_{t+h}-C_t =1\}}\Big] \\
= &\; \E_x\Big[\varphi\big(\chi(t)/2, \tau_{\vartheta_{C_t+1}}\big)\frac{2e^{\overline{{\mathcal V}}(t)}}{\chi(t)}{\bf 1}_{\{C_{t+h}-C_t =1\}}\Big]+\epsilon_2(h) \\
= &\; \E_x\Big[\varphi\big(\chi(t)/2, \tau_{\vartheta_{C_t+1}}\big)\frac{2e^{\overline{{\mathcal V}}(t)}}{\chi(t)}{\bf 1}_{\{C_{t+h}-C_t \geq 1\}}\Big]+\epsilon_3(h)
\end{align*}
where $\epsilon_2(h), \epsilon_3(h)\rightarrow 0$ as $h\rightarrow 0$, and where we used the second part of Lemma \ref{counting property} in order to obtain the last equality. Conditioning with respect to ${\mathcal F}_t \bigvee \tau_{\vartheta_{C_t+1}}$ and using that
$\{C_{t+h}-C_t \geq 1\}$ and $\tau_{\vartheta_{C_t+1}}$ are independent, applying the first part of Lemma \ref{counting property}, this last term is equal to
\begin{align*}
&  \E_x\Big[\varphi\big(\chi(t)/2, \tau_{\vartheta_{C_t+1}}\big)\frac{2e^{\overline{{\mathcal V}}(t)}}{\chi(t)}B\big(\chi(t)\big)h\Big] + \epsilon_4(h)\\
=\;&  \E_x\Big[\int_{{\mathcal E}}\varphi\big(\chi(t)/2, v'\big)\rho\big({\mathcal V}(t), dv'\big)\frac{2e^{\overline{{\mathcal V}}(t)}}{\chi(t)}B\big(\chi(t)\big)h\Big] + \epsilon_4(h)
\end{align*}
where $\epsilon_4(h)\rightarrow 0$ as $h \rightarrow 0$. Finally, using Lemma \ref{counting property} again, we derive
\begin{align}
& \E_x\Big[\Delta_h\Big(\varphi\big(\chi(t), {\mathcal V}(t)\big)\frac{e^{\overline{{\mathcal V}}(t)}}{\chi(t)}\Big){\bf 1}_{\{C_{t+h}-C_t =1\}}\Big] \nonumber \\
\rightarrow &\;\E_x\Big[\Big(\int_{{\mathcal E}}2\varphi\big(\chi(t)/2,v'\big)\rho\big({\mathcal V}(t), dv'\big)-\varphi\big(\chi(t),{\mathcal V}(t)\big)\Big)\frac{e^{\overline{{\mathcal V}}(t)}}{\chi(t)}B\big(\chi(t)\big)\Big]
\label{convergence poisson 1}
\end{align}
as $h \rightarrow 0$. By Proposition \ref{many-to-one}, this last quantity is equal to
$$\big\langle n(t,dx,dv), \big(\int_{{\mathcal E}}2\varphi(x/2,v')\rho(v,dv')-\varphi(x,v)\big)B(x)\big\rangle$$
which, in turn,  is equal to 
$$\big\langle n(t,2dx,dv),\int_{{\mathcal E}}4\varphi(x,v')\rho(v,dv')B(2x)\big\rangle-\big\langle n(t,dx,dv),\varphi(x,v)B(x)\big\rangle$$
by a simple change of variables. Putting together the estimates \eqref{first estimate poisson}, \eqref{convergence poisson 2} and \eqref{convergence poisson 1}, we conclude
\begin{align*}
&\partial_t \langle n(t,dx,dv), \varphi\rangle - \langle n(t,dx,dv), xv\partial_x \varphi\rangle + \langle n(t,dx,dv)B(x),\varphi\rangle \\
 = & \;  \big\langle n(t,2dx,dv),\int_{{\mathcal E}}4\varphi(x,v')\rho(v,dv')B(2x)\big\rangle,
\end{align*}
which is the dual formulation of \eqref{transport variabilite}. The proof is complete.
\subsection{Geometric ergodicity of the discrete model} \label{sec geo ergo}
We keep up with the notations of Sections \ref{microscopic model} and \ref{statistical analysis}. We first prove Proposition \ref{rep B via inv}.
\subsubsection*{Proof of Proposition \ref{rep B via inv}}
The fact that $\nu_B(d{\boldsymbol x}) = \nu_B(x,dv)dx$ readily follows from the representation ${\mathcal P}_{B}(\boldsymbol{x},d\boldsymbol{x'})  ={\mathcal P}_B\big((x,v), x', dv')dx'$ together with the invariant measure equation \eqref{def mesure invariante}. It follows that for every $y \in (0,\infty)$,
\begin{align*}
\nu_B(y, dv') & = \int_{{\mathcal S}} \nu_B(x,dv)dx\,{\mathcal P}_B\big((x,v),y,dv'\big) \\
& = \frac{B(2y)}{y} \int_{{\mathcal E}}\int_{0}^{2y}\nu_B(x, dv)\exp\big(-\int_{x/2}^y\tfrac{B(2s)}{v s}ds\big)\frac{\rho(v,dv')}{v}dx.
\end{align*}
By Assumption \ref{basic assumption}, we have $\int_{x/2}^\infty \tfrac{B(2s)}{s}ds=\infty$ hence
$$\exp\big(-\int_{x/2}^y \tfrac{B(2s)}{v s}ds\big)= \int_{y}^\infty\tfrac{B(2s)}{v s}\exp\big(-\int_{x/2}^s\tfrac{B(2s')}{v s'}ds'\big)ds,$$
therefore $\nu_B(y,dv')$ is equal to 
\begin{align*}
\,&\frac{B(2y)}{y} \int_{{\mathcal E}}\int_{0}^{2y}\nu_B(x, dv)dx \int_{y}^\infty\tfrac{B(2s)}{v s}\exp\big(-\int_{x/2}^s\tfrac{B(2s')}{v s'}ds'\big)ds\frac{\rho(v,dv')}{v} \\
=\;& \frac{B(2y)}{y} \int_{{\mathcal S}}\int_{[0,\infty)}{\bf 1}_{\displaystyle \{x\leq 2y, s \geq y\}}v^{-1} \nu_B(x,dv)dx\,{\mathcal P}_B\big((x,v),s,dv'\big)ds. 
\end{align*}
Integrating with respect to $dv'$, we obtain the result.
\subsubsection*{Geometric ergodicity}
We extend ${\mathcal P}_B$ as an operator acting on functions $f:{\mathcal S}\rightarrow [0,\infty)$ via
$${\mathcal P}_Bf(\boldsymbol{x}) = \int_{{\mathcal S}}f(\boldsymbol{y}){\mathcal P}_B(\boldsymbol{x},d\boldsymbol{y})$$
If $k\geq 1$ is an integer, define ${\mathcal P}_B^k={\mathcal P}_B^{k-1}\circ {\mathcal P}_B$.
\begin{prop} \label{prop transition}
Let $\mathfrak{c}$ satisfy Assumption \ref{contrainte constante}. Then, for every $B \in {\mathcal F}^\lambda(\mathfrak{c})$ and $\rho \in {\mathcal M}(\rho_{\min})$, there exists a unique invariant probability measure of the form $\nu_B(d\boldsymbol{x}) = \nu_B(x,dv)dx$ on ${\mathcal S}$.
Moreover, there exist $0<\gamma<1$, a function ${\mathbb V}: {\mathcal S} \rightarrow [1,\infty)$ and a constant $R$ such that
\begin{equation} \label{contraction}
\sup_{B \in {\mathcal F}^\lambda(\mathfrak{c}), \rho \in {\mathcal M}(\rho_{\min})}\sup_{|g| \leq V}\big|{\mathcal P}^k_Bg(\boldsymbol{x})-\int_{{\mathcal S}}g(\boldsymbol{z})\nu_B(d\boldsymbol{z})\big| \leq R{\mathbb V}(\boldsymbol{x})\gamma^k
\end{equation}
for every $\boldsymbol{x}\in {\mathcal S}$, $k\geq 0$, and where the supremum is taken over all functions $g: {\mathcal S} \rightarrow \R$ satisfying $|g(\boldsymbol{x})| \leq {\mathbb V}(\boldsymbol{x})$ for all $\boldsymbol{x}\in {\mathcal S}$. 
Moreover, under Assumption \ref{the full tree assumption}, we can take $\gamma < \frac{1}{2}$.
Finally, the function ${\mathbb V}$ is $\nu_B$-integrable for every $B \in {\mathcal F}^{\lambda}({\mathfrak{c}})$ and \eqref{contraction} is well defined.
\end{prop}
We will show in the proof that the function $\mathbb{V}$ defined in  \eqref{first def Lyapu} satisfies the properties announced in Proposition \ref{prop transition}.
\subsubsection*{Proof of Proposition \ref{prop transition}}
We follow the classical line of establishing  successively  a condition of minorisation, strong aperiodicity and drift for the transition operator ${\mathcal P}_B$ (see for instance
\cite{MT,B, FMP}. We keep in with the notation of Baxendale \cite{B}). Recall that $0<e_{\min}\leq e_{\max}$ is such that ${\mathcal E} \subset [e_{\min},e_{\max}]$.
\begin{proof}[Minorisation condition] Let $B \in {\mathcal F}^\lambda(\mathfrak{c})$. Define
\begin{equation} \label{def varphi}
\varphi_B(y) = \frac{B(2y)}{e_{\max} y}\exp\big(-\int_0^y \tfrac{B(2s)}{e_{\min} s}ds\big).
\end{equation}
Set ${\mathcal C} = (0, r)\times {\mathcal E}$, where $r$ is specified by $\mathfrak{c}$.
For any measurable $\mathcal{X} \times A \subset {\mathcal S}$ and $(x,v) \in {\mathcal C}$, we have
\begin{align*}
{\mathcal P}_B\big((x,v),\mathcal{X}\times A\big) & = \int_{A}\rho(v,dv')\int_{{\mathcal X} \cap [x/2,\infty]} \frac{B(2y)}{v y}\exp\big(-\int_{x/2}^y \tfrac{B(2s)}{v s}ds\big)dy \\
& \geq \rho_{\min}(A)\int_{{\mathcal X} \cap [r/2,\infty]} \varphi_B(y)dy.
\end{align*}
Define 
$$\Gamma_B(dy,dv)=c_B^{-1}{\bf 1}_{[r/2,\infty)}(y)\varphi_B(y)dy\,\tfrac{\rho_{\min}(dv)}{\rho_{\min}({\mathcal E})},$$
where
\begin{align*}
c_B  = \rho_{\min}({\mathcal E})\int_{r/2}^\infty \varphi_B(y)dy 
\geq \frac{e_{\min}\rho_{\min}({\mathcal E})}{e_{\max}}\exp\big(-\tfrac{L}{e_{\min}}) =: \widetilde \beta >0
\end{align*}
by \eqref{loc control} since $B \in {\mathcal F}^\lambda(\mathfrak{c})$.
We have thus
exhibited a small set ${\mathcal C}$, a probability measure $\Gamma_B$ and a constant $\widetilde \beta >0$ so that the minorisation condition
\begin{equation} \label{minorisation}
{\mathcal P}_B\big((x,v),{\mathcal X}\times A\big) \geq \tilde \beta\, \Gamma_B({\mathcal X} \times A)
\end{equation}
holds for every $(x,v)\in {\mathcal C}$ and ${\mathcal X}\times A \subset {\mathcal S}$, uniformly in $B \in {\mathcal F}^\lambda(\mathfrak{c})$.
\end{proof}
\begin{proof}[Strong aperiodicity condition] We have
\begin{align}
\tilde \beta\, \Gamma_B({\mathcal C})& = \widetilde \beta\, c_B^{-1}\int_{{\mathcal E}}\rho_{\min}(dv)\int_{r/2}^r \frac{B(2y)}{e_{\max} y}\exp\big(-\int_{x/2}^y \tfrac{B(2s)}{e_{\min} s}ds\big)dy  \nonumber\\
& \geq \widetilde \beta\, c_B^{-1}\int_{r/2}^r \varphi_B(y)dy \nonumber\\
& \geq \widetilde \beta \big(1-\exp\big(-\int_{r/2}^{r}\tfrac{B(2y)}{e_{\min} y}dy\big)\big) \nonumber\\
& \geq \widetilde \beta (1-\exp(-\tfrac{\ell}{e_{\min}})\big) =: \beta >0 \label{aperiodicity}
\end{align}
where we applied \eqref{loc control} for the last inequality.
\end{proof}
\begin{proof}[Drift condition] Let $B \in {\mathcal F}^\lambda(\mathfrak{c})$. 
Let ${\mathbb V}:{\mathcal S}\rightarrow [1,\infty)$ be continuously differentiable and such that for every $v\in {\mathcal E}$,
\begin{equation} \label{control V alinfini}
\lim_{y \rightarrow \infty} {\mathbb V}(y,v)\exp\big(-2^\lambda \tfrac{m}{v \lambda}y^\lambda\big)=0.
\end{equation}
For $x \geq r$, by \eqref{poly control} and integrating by part with the boundary condition \eqref{control V alinfini}, we have, for every $v \in {\mathcal E}$, 
\begin{align}
{\mathcal P}_B{\mathbb V}(x,v)  & = \int_{{\mathcal E}}\rho(v,dv') \int_{x/2}^\infty {\mathbb V}(y,v')\frac{B(2y)}{v y} \exp\big(-\int_{x/2}^y \tfrac{B(2s)}{v s}ds\big)dy \nonumber \\
& \leq  \int_{\mathcal {E}}\rho(v,dv') \int_{x/2}^\infty \partial_y {\mathbb V}(y,v') \exp\big(-\tfrac{m2^\lambda}{v}\int_{x/2}^y s^{\lambda-1}ds\big)dy \nonumber \\ 
& \leq \exp\big(\tfrac{m}{v\lambda}x^\lambda\big) \int_{\mathcal {E}}\rho(v,dv')\int_{\tfrac{m2^\lambda}{v \lambda}(x/2)^\lambda}^\infty {\mathbb V}\Big(\big(y\tfrac{v \lambda}{m 2^\lambda}\big)^{1/\lambda},v'\Big)e^{-y}dy. \nonumber
\end{align}
Pick ${\mathbb V}(x,v)={\mathbb V}(x)=\exp\big(\tfrac{m}{e_{\min} \lambda}x^\lambda\big)$ defined in \eqref{first def Lyapu}
and note that \eqref{control V alinfini} is satisfied for an apprioriate choice of $e_{\min}$ and $e_{\max}$ since $2^\lambda > \sup \mathcal E/\inf \mathcal E$.
With this choice, we further infer
\begin{align*}
{\mathcal P}_B{\mathbb V}(x,v) & \leq {\mathbb V}(x,v) \int_{\mathcal {E}}\rho(v,dv') \int_{\tfrac{m2^\lambda}{v \lambda}(x/2)^\lambda}^\infty \exp\big(-(1-2^{-\lambda})y\big)dy \\ 
& \leq {\mathbb V}(x,v) \frac{1}{1-2^{-\lambda}} \exp\big(- (1-2^{-\lambda})\tfrac{m}{v \lambda} r^\lambda\big)\rho_{\max}({\mathcal E})
\end{align*}
since $x \geq r$. Recall that
$$
\delta(\mathfrak{c})= \frac{1}{1-2^{-\lambda}} \exp\big(- (1-2^{-\lambda})\tfrac{m}{e_{\max} \lambda} r^\lambda\big)\rho_{\max}({\mathcal E})
.
$$
We obtain, for $x\geq r$ and $v \in {\mathcal E}$
\begin{equation} \label{first drift}
{\mathcal P}_B{\mathbb V}(x,v) \leq \delta(\mathfrak{c}) {\mathbb V}(x,v)
\end{equation}
and 
we have $\delta(\mathfrak{c}) < 1$ by Assumption \ref{contrainte constante}.
We next need to control ${\mathcal P}_B{\mathbb V}$ outside $x \in [r,\infty)$, that is on the small set ${\mathcal C}$.
For every $(x,v) \in {\mathcal C}$, we have
\begin{align}
 {\mathcal P}_B{\mathbb V}(x,v) 
 \leq &\int_{{\mathcal E}}\rho(v,dv')\Big(\int_{x/2}^{r/2} {\mathbb V}(y,v') \frac{B(2y)}{vy}dy \nonumber \\
 &+ \int_{r/2}^\infty {\mathbb V}(y,v') \frac{B(2y)}{v y}\exp\big(-\int_{r/2}^y \tfrac{B(2s)}{v s}ds\big)dy\Big)\nonumber \\
 \leq &\,\rho_{\max}({\mathcal E})
\Big(e_{\min}^{-1}\sup_{y \in [0,r]}{\mathbb V}(y) L+\delta(\mathfrak{c}) {\mathbb V}(r/2) \Big)=: Q <\infty \label{second drift} 
\end{align}
where we used \eqref{loc control} and the fact that $B \in {\mathcal F}^\lambda(\mathfrak{c})$. 
Combining together \eqref{first drift} and \eqref{second drift}, we conclude
\begin{equation} \label{drift condition}
{\mathcal P}_B{\mathbb V}(\boldsymbol{x}) \leq \delta(\mathfrak{c}) {\mathbb V}(\boldsymbol{x}) {\bf 1}_{\{\boldsymbol{x} \notin {\mathcal C}\}}+ Q{\bf 1}_{\{\boldsymbol{x}\in {\mathcal C}\}}. 
\end{equation}
\end{proof}
\begin{proof}[Completion of proof of Proposition \ref{prop transition}]
The minorisation condition \eqref{minorisation} together with the strong aperiodicity condition \eqref{aperiodicity} and the drift condition \eqref{drift condition} imply inequality \eqref{contraction} by Theorem 1.1 in Baxendale \cite{B}, with $R$ and $\gamma $ that explicitly depend on $\delta(\mathfrak{c})$, $\beta$, $\tilde \beta$, $V$ and $Q$. By construction, this bound is uniform in $B \in {\mathcal F}^\lambda(\mathfrak{c})$ and $\rho \in {\mathcal M}(\rho_{\min})$. More specifically, we have
$$\gamma < \min\{\max\{\delta({c}), \gamma_{{\mathbb V},B}\}, 1\}$$ 
therefore under Assumption \ref{contrainte constante} we have $\gamma <1$ and under Assumption \ref{the full tree assumption}, we obtain the improvement $\gamma < \frac{1}{2}$.
\end{proof}
\subsection{Further estimates on the invariant probability} \label{further invariant}
\begin{lemma} \label{borne sup nu}
For any $\mathfrak{c}$ such that Assumption \ref{contrainte constante} is satisfied and any compact interval ${\mathcal D} \subset (0,\infty)$, we have
$$\sup_{B \in {\mathcal F}^\lambda(\mathfrak{c})\, \cap\, {\mathcal H}^s({\mathcal D}, M)}\sup_{x \in {2^{-1}\mathcal D}}\nu_B(x)<\infty,$$
with $\nu_B(x) = \int_{{\mathcal E}}\nu_B(x,dv)$.
\end{lemma}
\begin{proof}
Since $B  \in {\mathcal F}^\lambda(\mathfrak{c})$, $\nu_B$ is well-defined and satisfies
\begin{align*}
\nu_B(x,dv) & = \frac{B(2x)}{x} \int_{{\mathcal E}}\int_{0}^{2x}\nu_B(y, dv')dy \exp\big(-\int_{y/2}^x\tfrac{B(2s)}{v' s}ds\big)\frac{\rho(v',dv)}{v'}.
\end{align*}
Hence $\nu_B(x,dv) \leq B(2x)(e_{\min}x)^{-1}\rho_{\max}(dv)$
and we also have
$\nu_B(x) \leq B(2x)(e_{\min}x)^{-1}\rho_{\max}({\mathcal E})
$.
Since $B \in {\mathcal H}^s({\mathcal D},M)$ implies $\sup_{x\in 2^{-1}{\mathcal D}}B(2x) = \|B\|_{L^{\infty}({\mathcal D})} \leq M$, the conclusion follows.
\end{proof}
\begin{lemma} \label{minoration mes inv}
For any $\mathfrak{c}$ such that Assumption \ref{contrainte constante} is satisfied, there exists a constant $d(\mathfrak{c}) \geq 0$ such that for any compact interval ${\mathcal D}\subset (d(\mathfrak{c}),\infty)$, we have
$$\inf_{B \in {\mathcal F}^\lambda(\mathfrak{c})}\inf_{x \in {\mathcal D}}\varphi_B(x)^{-1}\nu_B(x) >0,$$
where $\varphi_B(x)$ is defined in \eqref{def varphi}.
\end{lemma}
\begin{proof} Let $g:[0,\infty)\rightarrow [0,\infty)$ satisfy $g(x) \leq {\mathbb V}(x)=\exp(\tfrac{m}{e_{\min} \lambda}x^{\lambda})$ for every $x \in [0,\infty)$. By Proposition \ref{prop transition}, we have
\begin{equation} \label{integ unif mes inv}
\sup_{B\in \mathcal{F}_\lambda^\nu(\mathfrak{c})}\int_{[0,\infty)}g(x)\nu_B(x)dx<\infty,
\end{equation}
as a consequence of \eqref{contraction} with $n=1$ together with the property that $\sup_{B\in \mathcal{F}_\lambda^\nu(\mathfrak{c})}{\mathcal P}_B{\mathbb V}(\boldsymbol{x})<\infty$ for every $\boldsymbol{x} \in \mathcal{S}$, as follows from \eqref{drift condition} in the proof of Proposition \ref{prop transition}. Next, for every $x\in (0,\infty)$, we have
\begin{align*}
\int_{2x}^\infty \nu_B(y)dy \;
&\leq \exp(-\tfrac{m}{e_{\min} \lambda}(2x)^{\lambda})\int_{[0,\infty)}{\mathbb V}(y)\nu_B(y)dy
\end{align*}
and this bound is uniform in $B\in \mathcal{F}^\lambda(\mathfrak{c})$ by \eqref{integ unif mes inv}. Therefore, for every $x\in (0,\infty)$, we have
\begin{equation} \label{ineg qui tue}
\sup_{B\in \mathcal{F}_\lambda^\nu(\mathfrak{c})} \int_{2x}^\infty \nu_B(y)dy \leq c(\mathfrak{c}) \exp(-\tfrac{m}{e_{\min} \lambda}(2x)^{\lambda})
\end{equation}
for some $c(\mathfrak{c})>0$. Let 
\begin{equation} \label{def d}
d(\mathfrak{c})>\big(\frac{e_{\min} \lambda 2^{-\lambda}}{m} \log c(\mathfrak{c})\big)^{1/\lambda} {\bf 1}_{\{c(\mathfrak{c}) \geq 1\}}.
\end{equation}
By definition of $\nu_B$, for every $x\in (0,\infty)$, we now have
\begin{align*}
\nu_B(x,dv) & = \frac{B(2x)}{x}\int_{\mathcal{E}}\int_{0}^{2x} \nu_B(y,dv')\exp\big(-\int_{y/2}^x\tfrac{B(2s)}{v' s}ds\big)dy\frac{\rho(v', dv)}{v'} \\
& \geq  \frac{B(2x)}{e_{\max} x}\exp\big(-\int_{0}^x\tfrac{B(2s)}{e_{\min} s}ds\big)\int_0^{2x} \nu_B(y)dy\, \rho_{\min}(dv)\\
& \geq  \frac{B(2x)}{e_{\max} x}\exp\big(-\int_{0}^x\tfrac{B(2s)}{e_{\min} s}ds\big)\Big(1-c(\mathfrak{c}) \exp(-\tfrac{m}{e_{\min} \lambda}(2x)^{\lambda})\Big)\rho_{\min}(dv) 
\end{align*}
where we used \eqref{ineg qui tue} for the last inequality. By \eqref{def d}, for $x \geq d(\mathfrak{c})$ we have 
$$\big(1-c(\mathfrak{c}) \exp(-\tfrac{m}{e_{\min} \lambda}(2x)^{\lambda})\big)>0$$ and the conclusion readily follows by integration.
\end{proof}
\subsection{Covariance inequalities} \label{covariance inequalities}
If $u, w \in {\mathcal U}$, we define $a(u,w)$ as the node of the most recent common ancestor between $u$ and $w$. Introduce the distance 
$$\mathbb{D}(u,w) = |u|+|w|-2|a(u,w)|.$$
\begin{prop} \label{moment bound} 
Work under Assumption \ref{contrainte constante}.  
Let $\mu$ be a probability distribution on ${\mathcal S}$ such that $\int_{\mathcal {S}}{\mathbb V}(\boldsymbol{x})^2\mu(d\boldsymbol{x})<\infty$.
Let $G: {\mathcal S}\rightarrow \R$ and $H:[0,\infty)\rightarrow \R$ be two bounded functions.
Define
$$Z(\xi_{u^-},\tau_{u^-}, \xi_u) = G(\xi_{u^-},\tau_{u^-})H(\xi_u) - \E_{\nu_B}[G(\xi_{u^-},\tau_{u^-})H(\xi_u)].$$
For any $u,w\in {\mathcal U}$ with $|u|,|w| \geq 1$, we have
\begin{align}
\big| \E_{\mu}\big[Z(\xi_{u^-},\tau_{u^-}, \xi_u) Z(\xi_{w^-},\tau_{w^-}, \xi_w) \big]  \big|
\lesssim \gamma^{\,\mathbb{D}(u,w)} \label{covariance GH}
\end{align}
uniformly in $B \in {\mathcal F}^\lambda(\mathfrak{c})$, with $\gamma$ and $\nu_B$ defined in \eqref{contraction} of Proposition \ref{prop transition}. 
\end{prop}
\begin{proof}
In view of \eqref{covariance GH}, with no loss of generality, we may (and will) assume that for every $(x,v)\in {\mathcal S}$
\begin{equation} \label{controle lyapou}
|G(x,v)|\leq {\mathbb V}(x)\;\;\text{and}\;\;|H(x)| \leq {\mathbb V}(x)
\end{equation}
Applying repeatedly the Markov property along the branch that joins the nodes $a^-(u,w):=a(u^-,w^-)$ and $w$, we have
\begin{align*}
&\E_{\mu}\big[G(\xi_{u^-}, \tau_{u^-})H(\xi_u)|\,\xi_{a^-(u,w)}, \tau_{a^-(u,w)}\big]\\
 = &\;{\mathcal P}_B^{|u^-|-|a^-(u,w)|}(G\,{\mathcal P}_B H)\big(\xi_{a^-(u,w)},\tau_{a^-(u,w)}\big) \\
 =&\; {\mathcal P}_B^{|u|-|a(u,w)|}(G\,{\mathcal P}_B H)\big(\xi_{a^-(u,w)},\tau_{a^-(u,w)}\big)
\end{align*}
with an analogous formula for $G(\xi_{w^-}, \tau_{w^-})H(\xi_w)$.
Conditioning with respect to $\xi_{a^-(u,w)}, \tau_{a^-(u,w)}$, it follows that
\begin{align*}
&\E_{\mu}\big[Z(\xi_{u^-},\tau_{u^-}, \xi_u) Z(\xi_{w^-},\tau_{w^-}, \xi_w) \big] \\
=\;&  \E_{\mu}\Big[ \Big({\mathcal P}_B^{|u|-|a(u,w)|}(G\,{\mathcal P}_B H)(\xi_{a^-(u,w)}, \tau_{a^-(u,w)})-\E_{\nu_B}[G\,{\mathcal P}_B H(\xi_{\emptyset}, \tau_{\emptyset})]\Big)\\
&\;\hspace{1cm}\Big({\mathcal P}_B^{|w|-|a(u,w)|}(G\,{\mathcal P}_B H)(\xi_{a^-(u,w)}, \tau_{a^-(u,w)})-\E_{\nu_B}[G\,{\mathcal P}_B H(\xi_{\emptyset}, \tau_{\emptyset})]\Big)\Big].
\end{align*}
Applying Proposition \ref{prop transition} thanks to Assumption \ref{contrainte constante} and \eqref{controle lyapou}, we further infer
\begin{align*}
 \E_{\mu}\big[Z(\xi_{u^-},\tau_{u^-}, \xi_u) Z(\xi_{w^-},\tau_{w^-}, \xi_w)\big]  
\leq & \,R^2\sup_x H(x)^2\E_{\mu}\big[{\mathbb V}(\xi_{a^-(u,w)})^2\big]\gamma^{\,\mathbb{D}(u,w)}\\
\lesssim \; & \int_{\mathcal{S}}{\mathcal P}_B^{|a^-(u,w)|}\big({\mathbb V}^2\big)(\boldsymbol{x})\mu(d\boldsymbol{x})\, \gamma^{\,\mathbb{D}(u,w)}.
\end{align*}
We leave to the reader the straightfoward task to check that the choice of ${\mathbb V}$ in \eqref{first def Lyapu} implies that  ${\mathbb V}^2$ satisfies \eqref{drift condition}. It follows that Proposition \ref{prop transition} applies, replacing ${\mathbb V}$ by ${\mathbb V}^2$ in \eqref{contraction}. In particular,
\begin{equation} \label{majoration V deux}
\sup_{B \in {\mathcal F}^\lambda(\mathfrak{c})}{\mathcal P}_B^{|a^-(u,w)|}\big({\mathbb V}^2\big)(\boldsymbol{x}) \lesssim 1+{\mathbb V}(\boldsymbol{x})^2.
\end{equation}
Since ${\mathbb V}^2$ is $\mu$-integrable by assumption, inequality \eqref{covariance GH} follows. 
\end{proof} 
\begin{prop} \label{moment bound bis}
Work under Assumption \ref{contrainte constante}. 
Let $\mu$ be a probability on ${\mathcal S}$ such that $\int_{\mathcal{S}}{\mathbb V}(\boldsymbol{x})^2\mu(d\boldsymbol{x})<\infty$.
Let $x_0$ be in the interior of $\frac{1}{2}{\mathcal D}$. 
Let $H:\R\rightarrow \R$ be bounded with compact support. Set
$$\widetilde H\big(\tfrac{\xi_u-x_0}{h}\big) = H\big(\tfrac{\xi_u-x_0}{h}\big)-\E_{\nu_B}\big[H\big(\tfrac{\xi_\emptyset-x_0}{h}\big)\big].$$
For any $u,w\in {\mathcal U}$ with $|u|, |w| \geq 1$, we have
\begin{align} 
\big| \E_{\mu}\big[\widetilde H\big(\tfrac{\xi_u-x_0}{h}\big)\widetilde H\big(\tfrac{\xi_w-x_0}{h}\big)\big]  \big| 
\lesssim \; &  \gamma^{\,\mathbb{D}(u,w)} \bigwedge h 
\gamma^{\,\mathbb{D}(u,a(u,w)) \vee \,\mathbb{D}(w,a(u,w))} 
\label{first covariance}
\end{align}
uniformly in $B \in {\mathcal F}^\lambda(\mathfrak{c}) \cap {\mathcal H}^s({\mathcal D}, M)$ for sufficiently small $h>0$.
\end{prop}
\begin{proof} The first part of the estimate in the right-hand side of \eqref{first covariance} is obtained by letting $G=1$ in \eqref{covariance GH}.
We turn to the second part. 
Repeating the same argument as for \eqref{covariance GH} and conditioning with respect to $\xi_{a(u,w)}$, we obtain
\begin{align}
&
\E_{\mu}\big[\widetilde H\big(\tfrac{\xi_u-x_0}{h}\big)\widetilde H\big(\tfrac{\xi_w-x_0}{h}\big)\big]\nonumber\\
= \;& \E_{\mu}\Big[\big({\mathcal P}_B^{|u|-|a(u,w)|}H\big(\tfrac{\xi_{a(u,w)}-x_0}{h}\big)
-\E_{\nu_B}\big[H\big(\tfrac{\xi_\emptyset-x_0}{h}\big)\big]\big) \nonumber\\
& \hspace{3cm}
\big({\mathcal P}_B^{|w|-|a(u,w)|}H\big(\tfrac{\xi_{a(u,w)}-x_0}{h}\big)
-\E_{\nu_B}\big[H\big(\tfrac{\xi_\emptyset-x_0}{h}\big)\big]\big) \label{decomp markov}
\Big].
\end{align}
Assume with no loss of generality that $|u| \leq |w|$ (otherwise, the same subsequent arguments apply exchanging the roles of $u$ and $w$).
On the one hand, applying \eqref{contraction} of Proposition \ref{prop transition}, 
we have
\begin{equation} \label{premiere maj h}
\big|{\mathcal P}_B^{|w|-|a(u,w)|}H\big(\tfrac{\xi_{a(u,w)}-x_0}{h}\big)-\E_{\nu_B}\big[H\big(\tfrac{\xi_\emptyset-x_0}{h}\big)\big]
\big| \leq R{\mathbb V}\big(\xi_{a(u,w)}\big)\gamma^{|w|-|a(u,w)|}.
\end{equation}
On the other hand, identifying $H$ as a function defined on $\mathcal{S}$, for every $(x,v)\in {\mathcal S}$, we have
\begin{align}
\big|{\mathcal P}_BH\big(\tfrac{x-x_0}{h}\big)\big| & = \Big|\int_{x/2}^\infty H\big(h^{-1}(y-x_0)\big) \frac{B(2y)}{vy}\exp\big(-\int_{x/2}^y \tfrac{B(2s)}{vs}ds\big) dy\Big| \nonumber  \\
& \leq \int_{[0,\infty)} \big|H\big(h^{-1}(y-x_0)\big)\big| \frac{B(2y)}{e_{\min} y} dy   \nonumber\\
& \leq e_{\min}^{-1}\sup_{y\in\{x_0+h\,\mathrm{supp}(H)\}}\frac{B(2y)}{y} h \int_{[0,\infty)}|H(x)|dx  \lesssim h. \label{obtention h}
\end{align}
Indeed, since $x_0$ is in the interior of $\frac{1}{2}{\mathcal D}$ we have $\{x_0+h\,\mathrm{supp}(H)\} \subset \tfrac{1}{2}{\mathcal D}$  for small enough $h$ hence $\sup_{y\in\{x_0+h\,\mathrm{supp}(H)\}}B(2y) \leq M$. Now, since ${\mathcal P}_B$ is a positive operator and ${\mathcal P}_B {\bf 1} = {\bf 1}$, 
we derive
\begin{equation} \label{avant esperance h}
\big|{\mathcal P}_B^{|u|-|a(u,w)|}H\big(\tfrac{\xi_{a(u,w)}-x_0}{h}\big)\big| \lesssim h
\end{equation}
as soon as $|u|-|a(u,w)| \geq 1$, uniformly in $B \in {\mathcal F}^\lambda(\mathfrak{c}) \cap {\mathcal H}^s({\mathcal D}, M)$.  If $|u|=|a(u,w)|$, since
$\int_{\mathcal{E}}\nu_B(dx,dv) = \nu_B(x)dx$, we obtain in the same way
\begin{align}
\big|\E_{\nu_B}\big[H\big(\tfrac{\xi_{a(u,w)}-x_0}{h}\big)\big] \big|& \leq \int_{[0,\infty)}\big|H\big(\tfrac{x-x_0}{h}\big)\big|\nu_B(x)dx \lesssim h
\label{esperance h}
\end{align}
using Lemma \ref{borne sup nu}. We have $\big|\E_{\nu_B}\big[H\big(\tfrac{\xi_u-x_0}{h}\big)\big]\big| \lesssim h$ likewise.
Putting together \eqref{avant esperance h} and \eqref{esperance h}
we derive
$$\big|{\mathcal P}_B^{|u|-|a(u,w)|}H\big(\tfrac{\xi_{a(u,w)}-x_0}{h}\big)
-\E_{\nu_B}\big[H\big(\tfrac{\xi_u-x_0}{h}\big)\big]\big| \lesssim h$$
and this estimate is uniform in $B \in {\mathcal F}^\lambda(\mathfrak{c}) \cap {\mathcal H}^s({\mathcal D}, M)$. In view of  \eqref{decomp markov} and \eqref{premiere maj h}, we obtain
$$\E_{\mu}\big[\widetilde H\big(\tfrac{\xi_u-x_0}{h}\big)\widetilde H\big(\tfrac{\xi_w-x_0}{h}\big)\big]
\lesssim \;  h\gamma^{|w|-|a(u,w)|}\E_{\mu}\big[{\mathbb V}\big(\xi_{a(u,w)}\big)\big].
$$
We conclude in the same way as in Proposition \ref{moment bound}.
\end{proof}
\subsection{Rate of convergence for the empirical measure} \label{rate of convergence for the empirical measure}
For every $y\in (0,\infty)$ and $u\in {\mathcal U}$ with $|u|\geq 1$, define
\begin{equation} \label{def D}
D(y) = \E_{\nu_B}\big[\tfrac{1}{\tau_{u^-}}{\bf 1}_{\{\xi_{u^-}\leq 2y,\,\xi_u \geq y\}}\big],
\end{equation}
\begin{equation} \label{def D_n}
D_n(y)= n^{-1}\sum_{u \in {\mathcal U}_n} \tfrac{1}{\tau_{u^-}}{\bf 1}_{\displaystyle \{\xi_{u^-}\leq 2y, \xi_u \geq y\}},
\end{equation}
and
$$
D_n(y)_{\varpi} = D_n(y) \bigvee \varpi.
$$
\begin{prop} \label{moment lemma}
Work under Assumption \ref{contrainte constante} in the sparse tree case and Assumption \ref{the full tree assumption} in the full tree case. Let $\mu$ be a probability on ${\mathcal S}$ such that $\int_{\mathcal{S}}{\mathbb V}(\boldsymbol{x})^2\mu(d\boldsymbol{x})<\infty$.
If $1 \geq \varpi = \varpi_n\rightarrow 0$ as $n \rightarrow \infty$, we have
\begin{equation} \label{convergence D_n}
\sup_{y\in {\mathcal D}}\E_{\mu}\big[\big(D_n(y)_{\varpi_n}-D(y)\big)^2\big] \lesssim n^{-1}
\end{equation}
uniformy in $B \in {\mathcal F}^\lambda(\mathfrak{c}) \cap {\mathcal H}^s(2^{-1}{\mathcal D}, M)$ and $\rho \in {\mathcal M}(\rho_{\min}, \rho_{\max})$.
\end{prop}
We first need the following estimate
\begin{lemma} \label{minoration D}
Work under Assumption \ref{contrainte constante}.
Let $d(\mathfrak{c})$ be defined as in Lemma \ref{minoration mes inv}. For every compact interval ${\mathcal D} \subset (d(\mathfrak{c}),\infty)$ such that $\inf {\mathcal D} \leq r/2$, we have
$$
\inf_{B \in {\mathcal F}^\lambda(\mathfrak{c}) \cap {\mathcal H}^s(2^{-1}{\mathcal D}, M)}\inf_{y\in {\mathcal D}}D(y) > 0.
$$
\end{lemma}
\begin{proof}[Proof of Lemma \ref{minoration D}] 
By \eqref{representation cle} and the definition of $\varphi_B$ in \eqref{def varphi}, we readily have
$$
D(y) =  \frac{1}{e_{\max}}\varphi_B(y)^{-1}\nu_B(y)\exp\big(-\int_0^y \tfrac{B(2s)}{e_{\min} s}ds\big).
$$
Since $B\in {\mathcal F}^\lambda(\mathfrak{c})\cap {\mathcal H}^s(2^{-1}{\mathcal D}, M)$, by applying \eqref{loc control} and \eqref{poly control} successively, we obtain
\begin{align*}
\int_{0}^{\sup {\mathcal D}} \frac{B(2s)}{e_{\min} s}ds & \leq e_{\min}^{-1}L+\int_{r/2}^{\sup \mathcal{D}} \frac{B(2s)}{e_{\min} s}ds \\
& \leq e_{\min}^{-1}(L + M \log \tfrac{\sup {\mathcal D}}{r/2}) < \infty
\end{align*} 
where we used that $\inf {\mathcal D} \leq r/2$. It follows that
$$
\inf_{y\in {\mathcal D}}\exp\big(-\int_0^y \tfrac{B(2s)}{e_{\min} s}ds\big) \geq \exp\big(-e_{\min}^{-1}(L + M \log \tfrac{\sup {\mathcal D}}{r/2}) \big)>0
$$
and Lemma \ref{minoration D} follows by applying Lemma \ref{minoration mes inv}.
\end{proof}
\begin{proof}[Proof of Proposition \ref{moment lemma}]
Since $D_n(y)$ is bounded, we have
\begin{equation} \label{take off varpi}
\big(D_n(y)_{\varpi_n}-D(y)\big)^2\lesssim \big(D_n(y)-D(y)\big)^2+{\bf 1}_{\{D_n(y) < \varpi_n\}}.
\end{equation}
Next, take $n$ sufficiently large, so that 
$$
0 < \varpi_n \leq q=\tfrac{1}{2}\inf_{B \in {\mathcal F}^\lambda(\mathfrak{c}) \cap {\mathcal H}^s(2^{-1}{\mathcal D}, M)}\inf_{y \in {\mathcal D}}D(y)
$$
a choice which is possible thanks to Lemma \ref{minoration D}.
Since 
$$\{D_n(y) < \varpi_n\}\subset \{D_n(y)-D(y)< -q\},$$
integrating \eqref{take off varpi}, we have that $\E_{\mu}\big[\big(D_n(y)_{\varpi_n}-D(y)\big)^2\big]$ is less than a constant times
\begin{align*}
 \E_{\mu}\big[\big(D_n(y)-D(y)\big)^2\big] + \PP_{\mu}\big(|D_n(y)-D(y)| \geq q\big),
\end{align*}
which in turn is less than a constant times $\E_{\mu}\big[\big(D_n(y)-D(y)\big)^2\big]$.
Set $G(x, v) = \frac{1}{v}{\bf 1}_{\{x \leq 2y\}}$ and $H(x)={\bf 1}_{\{x \geq y\}}$ and note that $G$ and $H$ are bounded on ${\mathcal S}$ (and also uniformly in $y \in {\mathcal D}$).
It follows that
\begin{align*}
D_n(y) -D(y) & = n^{-1}\sum_{u \in {\mathcal U}_n}\Big( G(\xi_{u^-}, \tau_{u^-})H(\xi_u)-\E_{\nu_B}\big[G(\xi_{u^-}, \tau_{u^-})H(\xi_u)\big]\Big). 
\end{align*}
We then apply \eqref{covariance GH} of Proposition \ref{moment bound} to infer, with the same notation that
\begin{align*}
& \E_{\mu}\big[\big(D_n(y)-D(y)\big)^2\big] \\
 =&\;n^{-2}\sum_{u,w \in {\mathcal U}_n} \E_{\mu}\big[Z(\xi_{u^-},\tau_{u^-},\xi_u)Z(\xi_{w^-},\tau_{w^-},\xi_w)\big] \\
\lesssim  &\; n^{-2}\sum_{u,w \in {\mathcal U}_n}\gamma^{\mathbb{D}(u,w)}
\end{align*}
uniformly in $y\in {\mathcal D}$ and $B \in {\mathcal F}^\lambda(\mathfrak{c})$. We further separate the sparse and full tree cases.
\begin{proof}[The sparse tree case] We have
$
\sum_{u,w \in {\mathcal U}_n}\gamma^{\mathbb{D}(u,w)} = \sum_{1 \leq |u|,|w| \leq n}\gamma^{||u|-|w||}$ and by Proposition \ref{prop transition}, this last quantity is of order $n$.
\end{proof}
\begin{proof}[The full tree case]
We have $n \sim 2^{N_n}$, where $N_n$ is the number of generations used to expand ${\mathcal U}_n$. We evaluate
\begin{equation} \label{the big sum}
\mathcal {N}_k=\sum_{|u|=k}\sum_{w \in\, {\mathcal U}_n} \gamma^{\mathbb{D}(u,w)}\;\;\text{for}\;\;k=0,\ldots, N.
\end{equation}
For $k=0$, we have
$${\mathcal N}_0=1+2\gamma + 4\gamma^2 +\ldots + 2^{N_n}\gamma^{N_n} = \tfrac{1-(2\gamma)^{N_n+1}}{1-2\gamma}=: \phi_\gamma(N_n).$$
Under Assumption \ref{the full tree assumption}, by Proposition \ref{prop transition}, we have $\gamma < \frac{1}{2}$ therefore $\phi_\gamma(N_n)$ is bounded as $n \rightarrow \infty$. 
For $k=1$, if we start with the node $u=(\emptyset, 0)$, then the contribution of its descendants in \eqref{the big sum} is given by $\phi_\gamma(N_n-1)$, to which we must add $\gamma$ for its ancestor corresponding to the node $u=\emptyset$ and also $\gamma\phi_\gamma(N_n)$ for the contribution of the second lineage of the node $u=\emptyset$. 
Finally, we must repeat the argument for the node $u=(\emptyset, 1)$. We obtain
$$\mathcal{N}_1 = 2\big(\phi_\gamma(N_n-1)+\gamma+\gamma^2\phi_\gamma(N_n-1)\big).$$
More generally, proceeding in the same manner, we derive 
\begin{align} 
\mathcal{N}_{k}  & =2^k\Big(\phi_\gamma(N_n-k)+\big(\gamma+\gamma^2\phi_\gamma(N_n-k)\big)+ \ldots \nonumber \\
&+ \gamma^i+\gamma^{i+1}\phi_\gamma\big(N_n-k+(i-1)\big)+ \ldots+\big(\gamma^k+\gamma^{k+1}\phi_\gamma(N_n-1)\big)\Big) \label{def Vk} 
\end{align}
for $k=1,\ldots, N_n$, and this last quantity is of order $2^k$. It follows that
$$\sum_{u,w \in \,{\mathcal U}_n}\gamma^{\mathbb{D}(u,w)} = \sum_{k=0}^{N_n} {\mathcal N}_k \lesssim \sum_{k=1}^{N_n} 2^k \lesssim 2^{N_n} \lesssim  n$$
and the conclusion follows likewise.
\end{proof}
Putting together the sparse and  full tree case, we obtain the proposition.
\end{proof}
Let $\widehat \nu_n(dy) = n^{-1}\sum_{u \in {\mathcal U}_n}\delta_{\xi_u}(dy)$ denote the empirical measure of the observation $(\xi_u, u \in {\mathcal U}_n)$.
\begin{prop}\label{moment lemma bis}
Work under Assumption \ref{contrainte constante} in the sparse tree case and Assumption \ref{the full tree assumption} in the full tree case. Let $\mu$ be a probability on ${\mathcal S}$ such that $\int_{\mathcal{S}}{\mathbb V}(\boldsymbol{x})^2\mu(d\boldsymbol{x})<\infty$.
We have
\begin{equation} \label{moment mes inv}
\sup_{y \in {\mathcal D}}\E_{\mu}\big[\big(K_{h_n}\star \widehat \nu_n(y)-K_{h_n} \star \nu_B(y)\big)^2\big] \lesssim |\log h_n| {(nh_n)}^{-1}
\end{equation}
uniformly in $B \in {\mathcal F}^\lambda(\mathfrak{c})$.
\end{prop}
\begin{proof}
We have, with the notation of Proposition \ref{moment bound bis}
\begin{align}
& \E_{\mu}\big[\big(K_{h_n}\star \widehat \nu_n(y)-K_{h_n} \star \nu_B(y)\big)^2\big] \nonumber \\
 =\, &(nh_n)^{-2}\E_{\mu}\Big[\Big(\sum_{u\in\, {\mathcal U}_n}K\big(\tfrac{\xi_u-y}{h_n}\big)-\E_{\nu_B}\big[K\big(\tfrac{\xi_u-y}{h_n}\big)\big]\Big)^2\Big] \nonumber\\
 =\, &(nh_n)^{-2}\sum_{u,w\,\in {\mathcal U}_n}\mathbb{E}_{\mu}\big[\widetilde K\big(\tfrac{\xi_u-y}{h_n}\big) \widetilde K\big(\tfrac{\xi_w-y}{h_n}\big)\big] \nonumber\\
\lesssim\, &(nh_n)^{-2} \sum_{u,w \in\, {\mathcal U}_n}  \gamma^{\,\mathbb{D}(u,w)} \bigwedge h_n \label{estimation dist arbre}
\gamma^{\,\mathbb{D}(u,a(u,w)) \vee \,\mathbb{D}(w,a(u,w))} 
\end{align}
by applying \eqref{first covariance} of Proposition \ref{moment bound bis}.
It remains to estimate \eqref{estimation dist arbre}.
\begin{proof}[The sparse tree case] We have $a(u,w)=u$ if $|u|\leq |w|$ and $a(u,w)=w$ otherwise. It follows that 
$$
\E_{\mu}\big[\big(K_{h_n}\star \widehat \nu_n(y)-K_{h_n} \star \nu_B(y)\big)^2\big] 
\lesssim n^{-2}h_n^{-1}\sum_{u,w \in \,{\mathcal U}_n} \gamma^{\,\mathbb{D}(u,w)},$$
and since $\sum_{u,w \in \,{\mathcal U}_n} \gamma^{\,\mathbb{D}(u,w)}=\sum_{1 \leq |u|,|w| \leq n}\gamma^{||u|-|w||}$ is of order $n$
as soon as $\gamma <1$, we obtain the result.
\end{proof}
\begin{proof}[The full tree case] The computations are a bit more involved. 
Let us evaluate
$$\sum_{|u|=k}\sum_{w \in {\mathcal U}_n}\gamma^{\mathbb{D}(u,w)}\bigwedge h_n\gamma^{\,\mathbb{D}(u,a(u,w)) \vee \,\mathbb{D}(w,a(u,w))}.$$
We may repeat the argument displayed in \eqref{def Vk} in order to evaluate the contribution of the term involving $\gamma^{\,\mathbb{D}(u,a(u,w))}$. However, in the estimate ${\mathcal N}_k$, each term $\gamma^i+\gamma^{i+1}\phi_\gamma\big(N_n-k+(i-1)\big)$ in formula \eqref{def Vk} may be replaced  by $h_n\big(\gamma^i+\gamma\phi_\gamma\big(N_n-k+(i-1)\big)\big)$ up to constants. This corresponds to the correction given by $h_n\gamma^{\,\mathbb{D}(u,a(u,w)) \vee \,\mathbb{D}(w,a(u,w))}$. 
As a consequence, we obtain 
\begin{align*}
& \sum_{|u|=k}\sum_{w \in {\mathcal U}_n}\gamma^{\mathbb{D}(u,w)}\bigwedge h_n\gamma^{\,\mathbb{D}(u,a(u,w)) \vee \,\mathbb{D}(w,a(u,w))} \\
\lesssim &\; 2^k \sum_{i=1}^k h_n\big(\gamma^i+\gamma\phi_\gamma(N_n-k+i-1)\big) \bigwedge \gamma^i\big(1+\gamma \phi_\gamma(N_n-k+i-1)\big) \\
\lesssim &\; 2^k \sum_{i=1}^k h_n \wedge \gamma^{i}.
\end{align*}
Define $k_n^\star= \lfloor \tfrac{|\log h_n|}{|\log \gamma|}\rfloor$. We readily derive
$$2^k\sum_{i=1}^k h_n \wedge \gamma^{i} = 2^k\big(\sum_{i=1}^{k_n^\star}h_n+\sum_{i=k_n^\star+1}^k \gamma^i\big) \lesssim 2^k h_n|\log h_n|,$$
ignoring the second term if $k_n^\star+1 \geq k$.
Going back to \eqref{estimation dist arbre}, it follows that
\begin{align*}
&(nh_n)^{-2} \sum_{u,w \in\, {\mathcal U}_n}  \gamma^{\,\mathbb{D}(u,w)} \bigwedge h_n \gamma^{\,\mathbb{D}(u,a(u,w)) \vee \,\mathbb{D}(w,a(u,w))}  \\
=\, &(nh_n)^{-2}  \sum_{k=0}^{N_n}\sum_{|u|=k}\sum_{v \in {\mathcal U}_n}   \gamma^{\,\mathbb{D}(u,w)} \bigwedge h_n \gamma^{\,\mathbb{D}(u,a(u,w)) \vee \,\mathbb{D}(w,a(u,w))} \\
\lesssim\, & (nh_n)^{-2} \sum_{k=0}^{N_n} 2^k h_n|\log h_n| \lesssim |\log h_n| (nh_n)^{-1}
\end{align*}
and the conclusion follows in the full case.
\end{proof}
Putting together the sparse and  full tree cases, we obtain the proposition.
\end{proof}
\subsection{Proof of Theorem \ref{upper bound}} \label{final proof}
From
$$\widehat B_n(2y)= y\frac{n^{-1}\sum_{u \in {\mathcal U}_n}K_{h_n}(\xi_u-y)}{n^{-1}\sum_{u \in {\mathcal U}_n} \frac{1}{\tau_{u^-}}{\bf 1}_{\displaystyle \{\xi_{u^-}\leq 2y, \xi_u \geq y\}} \bigvee \varpi_n}$$
and
$$B(2y) =y \frac{\nu_B(y)}{ \E_{\nu_B}\big[\tfrac{1}{\tau_{u^-}}{\bf 1}_{\{\xi_{u^-}\leq 2y,\,\xi_u \geq y\}}\big]},$$
we plan to use the following decomposition
$$\widehat B_n(2y)-B(2y) = y(I+II+III),$$
with
\begin{align*}
I & =\frac{K_{h_n}\star \nu_B(y)-\nu_B(y)}{D(y)}, \\
II & =\frac{K_{h_n}\star \widehat \nu_n(y)-K_{h_n}\star \nu_B(y)}{D_n(y)_{\varpi_n}},\\
III & =\frac{K_{h_n}\star \nu_B(y)}{D_n(y)_{\varpi_n} D(y)}\big(D(y)-D_n(y)_{\varpi_n}\big), 
\end{align*}
where $D(y)$ and $D_n(y)_{\varpi}$ are defined in \eqref{def D} and \eqref{def D_n} respectively.
It follows that
\begin{align*}
\|\widehat B_n-B\|_{L^2({\mathcal D})}^2  & = 2\int_{\tfrac{1}{2}{\mathcal D}}\big(\widehat B_n(2y)-B(2y)\big)^2dy 
 \lesssim  IV+V+VI,
\end{align*}
with
\begin{align*}
IV & = \int_{\tfrac{1}{2}{\mathcal D}}\big(K_{h_n}\star \nu_B(y)-\nu_B(y)\big)^2\tfrac{y^2}{D(y)^2}dy\\
V & =  \int_{\tfrac{1}{2}{\mathcal D}}\big(K_{h_n}\star \widehat \nu_n(y)-K_{h_n}\star \nu_B(y)\big)^2D_n(y)_{\varpi}^{-2}y^2dy\\
VI & = \int_{\tfrac{1}{2}{\mathcal D}} \big(D_n(y)_\varpi-D(y)\big)^2\big(K_{h_n}\star \nu_B(y)\big)^2\big(D_n(y)_\varpi D(y)\big)^{-2}y^2dy.
\end{align*}
\begin{proof}[The term IV] We get rid of the term $\tfrac{y^2}{D(y)^2}$ by Lemma \ref{minoration D} and the fact that ${\mathcal D}$ is bounded. By Assumption \ref{prop K} and classical kernel approximation, we have for every $0 < s \leq n_0$
\begin{equation} \label{biais mesure invariante}
IV \lesssim \|K_{h_n}\star \nu_B-\nu_B\big\|_{L^2(2^{-1}{\mathcal D})}^2 \lesssim |\nu_B|_{{\mathcal H}^s(2^{-1}{\mathcal D})}^2h_n^{2s}.
\end{equation}
\begin{lemma} \label{reg transfer} Let ${\mathcal D}\subset (0,\infty)$ be a compact interval. Let $B\in {\mathcal F}^\lambda(\mathfrak{c})$ for some ${\mathfrak{c}}$ satisfying Assumption \ref{contrainte constante}. We have
$$\|\nu_B\|_{{\mathcal H}^s(2^{-1}\mathcal{D})} \leq \psi\big(e_{\min},e_{\max}, {\mathcal D},\|B\|_{{\mathcal H}^s({\mathcal D})}\big)$$
for some continuous function $\psi$. 
\end{lemma}
\begin{proof}[Proof of Lemma \ref{reg transfer}] 
Define 
$$\Lambda_B(x,y) = \int_{\mathcal E}\frac{\nu_B(y,dv')}{v'}\exp\big(-\int_{y/2}^x\tfrac{B(2s)}{v's}ds\big).$$
If $B\in{\mathcal H}^s({\mathcal D})$, then $x\leadsto \Lambda_B(x,y) \in {\mathcal H}^s(2^{-1}{\mathcal D})$ for every $y \in [0,\infty)$, and we have
$$\|\Lambda_B(\cdot,y)\|_{{\mathcal H}^s(2^{-1}{\mathcal D})} \leq \psi_1(y,\|B\|_{{\mathcal H}^s({\mathcal D})}, e_{\min}, e_{\max})$$
for some continous function $\psi_1$. The result is then a consequence of the representation
$\nu_B(x) = \frac{B(2x)}{x}\int_0^{2x} \Lambda_B(x,y)dy$.
\end{proof}
Going back to \eqref{biais mesure invariante} we infer from Lemma \ref{reg transfer} that $|\nu_B|_{{\mathcal H}^s(2^{-1}\mathcal{D})}$ is bounded above by
a constant that depends on $e_{\min}$, $e_{\max}$, ${\mathcal D}$ and $\|B\|_{{\mathcal H}^s({\mathcal D})}$ only. It follows that 
\begin{equation} \label{controle IV}
IV \lesssim h_n^{2s}
\end{equation}
uniformly in $B \in {\mathcal H}^s({\mathcal D},M)$.
\end{proof}
\begin{proof}[The term V] We have
$$\E_{\mu}[V] \leq 
\varpi_n^{-2}|{\mathcal D}| \sup_{y \in 2^{-1}{\mathcal D}} y^2\E_{\mu}\big[\big(K_{h_n}\star \widehat \nu_n(y)-K_{h_n} \star \nu_B(y)\big)^2\big].$$
By \eqref{moment mes inv} of Proposition \ref{moment lemma bis} we derive
\begin{equation} \label{controle V}
\E_{\mu}[V] \lesssim \varpi_n^{-2}|\log h_n|(nh_n)^{-1}
\end{equation}
uniformly in $B \in {\mathcal F}^\lambda(\mathfrak{c})$.
\end{proof}
\begin{proof}[The term VI] First, by Lemma \ref{minoration D}, the estimate
$$\inf_{B \in {\mathcal F}^\lambda(\mathfrak{c})}\inf_{y\in 2^{-1}{\mathcal D}}D_n(y)_\varpi D(y) \gtrsim \varpi_n$$
holds. Next, 
\begin{align}
\sup_{y \in  2^{-1}{\mathcal D}} |K_{h_n}\star \nu_B(y)| & = \sup_{y \in  2^{-1}{\mathcal D}}\big|\int_{[0,\infty)}K_{h_n}(z-y)\nu_B(z)dz\big| \nonumber \\
& \leq \sup_{y \in 2^{-1}{\mathcal D}_{h_n}}\nu_B(y)\|K\|_{L^1([0,\infty))} \label{borne phi l1}
\end{align}
where $2^{-1}{\mathcal D}_{h_n} = \{y+z,\;y\in  2^{-1}{\mathcal D},\;z\in \text{supp}(K_{h_n})\} \subset \widetilde {\mathcal D}$, for some compact interval $\widetilde {\mathcal D}$ since $K$ has compact support by Assumption \ref{prop K}. By Lemma \ref{borne sup nu}, we infer that \eqref{borne phi l1} holds uniformly in $B \in {\mathcal F}^\lambda(\mathfrak{c})$. We derive
$$\E_{\mu}\big[VI\big] \lesssim \varpi_n^{-2}\sup_{y \in 2^{-1}{\mathcal D}}\E_{\mu}\big[\big(D_n(y)_{\varpi_n}-D(y)\big)^2\big].$$
Applying \eqref{convergence D_n} of Proposition \ref{moment lemma}, we conclude
\begin{equation} \label{controle VI}
\E_{\mu}\big[VI\big] \lesssim \varpi_n^{-2}n^{-1}
\end{equation}
uniformly in $B \in {\mathcal F}^\lambda(\mathfrak{c})$.
\end{proof}
\begin{proof}[Completion of proof of Theorem \ref{upper bound}] 
We put together the three estimates \eqref{controle IV}, \eqref{controle V} and \eqref{controle VI}. We obtain
\begin{align*}
\E_{\mu}\big[\|\widehat B_n-B\|_{L^2({\mathcal D})}^2\big] \lesssim  h_n^{2s}+\varpi_n^{-2}|\log h_n|(nh_n)^{-1}+ \varpi_n^{-2}n^{-1}
\end{align*}
uniformly in $B \in {\mathcal F}^\lambda(\mathfrak{c})\,\cap \,{\mathcal H}^s({\mathcal D},M)$. The choice $h_n \sim n^{-1/(2s+1)}$ and the fact that $\varpi_n^{-2}$ grows logarithmically in $n$ yields the rate $n^{-s/(2s+1)}$ up to log terms and the inessential supplementary multiplicative error factor $\varpi_n^{-1}$. The proof is complete.
\end{proof}
\section{Appendix} \label{appendix}
\subsection{Construction of the discrete model} \label{construction of the chain}
Fix an initial condition $\boldsymbol{x} = (x,v)\in {\mathcal S}$.
On a rich enough probability space, we consider a Markov chain on the binary tree $(\tau_u, u \in {\mathcal U})$ with transition $\rho(v,dv')$ and initial condition $v$:
if $u = (u_1,\ldots, u_k) \in {\mathcal U}$, we write $ui = (u_1,\ldots, u_k, i)$, $i=0,1$ for the two offsprings of $u$; we set $\tau_\emptyset=v$ and
$$\tau_{u0}\sim \rho(\tau_{u}, dv')\;\;\text{and}\;\;\tau_{u1}\sim  \rho(\tau_{u}, dv')$$ 
so that conditional on $\tau_{u}$, the two random variables $\tau_{u0}$ and $\tau_{u1}$ are independent.
We also pick a sequence of independent standard exponential random variables $\big({\bf e}_u, u \in {\mathcal U}\big)$, independent of $(\tau_u, u \in {\mathcal U})$.
The model  $\big((\xi_u, \tau_u), u \in {\mathcal U}\big)$ is then constructed recursively. We set
$$\xi_\emptyset=x,\;\;b_\emptyset=0,\;\;\tau_\emptyset = v\;\;\text{and}\;\;\zeta_\emptyset=F_{x,v}^{-1}({\bf e}_\emptyset)$$ 
where $F_{x,v}(t) = \int_0^t B\big(x\exp(vs)\big)ds$. For $u \in {\mathcal U}$ and $i=0,1$, we put
$$\xi_{u0}=\xi_{u1}=e^{\tau_u \zeta_u}\frac{\xi_u}{2},\;\;b_{u0}=b_{u1}=b_u+\zeta_u,\;\;\zeta_{ui}=F^{-1}_{\xi_{ui}, \tau_{ui}}({\bf e}_{ui}).$$
To each node $u \in {\mathcal U}$, we then associate the mark $(\xi_i, b_u, \zeta_u, \tau_u)$ of the size, date of birth, lifetime and growth rate respectively  of the individual labeled by $u$. One easily checks that Assumption \ref{basic assumption} guarantees that the model is well defined. In particular, since $B$ is locally bounded, we see that there is no accumulation of jumps almost-surely.
\subsection{Proof of Lemma \ref{counting property}} \label{A1}
Note first that 
$$\{C_{t+h}-C_t \geq 1\} = \{t < b_{\vartheta_{C_t}}+\zeta_{\vartheta_{C_t}}\leq t+h\}.$$
Since moreover $\xi_{\vartheta_{C_t}} = 
x\exp\big(\overline{{\mathcal V}}(b_{\vartheta_{C_t}})\big)2^{-C_t}$, it follows by \eqref{def div rate} that 
\begin{align*}
&\PP(C_{t+h}-C_t \geq 1\,|\,{\mathcal F}_t) \\
=& \int_{t-b_{\vartheta_{C_t}}}^{t+h-b_{\vartheta_{C_t}}}B\Big(\tfrac{xe^{\overline{{\mathcal V}}(b_{\vartheta_{C_t}})+s{\mathcal V}(s)}}{2^{C_t}}\Big)\exp\Big(-\int_0^s B\Big(\tfrac{xe^{\overline{{\mathcal V}}(b_{\vartheta_{C_t}})+s'{\mathcal V}(s')}}{2^{C_t}}
\Big)ds'\Big)ds.
\end{align*}
Introduce the quantity $B\big(xe^{\overline{{\mathcal V}}(b_{\vartheta_{C_t}})+{\mathcal V}(t)(t-b_{\vartheta_{C_t}})}2^{-C_t}\big)$ within the integral. Noting that $\overline{{\mathcal V}}(b_{\vartheta_{C_t}})+{\mathcal V}(t)(t-b_{\vartheta_{C_t}}) = \overline{{\mathcal V}}(t)$ we obtain the first part of the lemma thanks to the representation \eqref{rep chi} and the uniform continuity of $B$ over compact sets. For the second part, introduce the $({\mathcal F}_t)$-stopping time
$$\Upsilon_t = \inf\{s>t, C_s - C_t \geq 1\}$$
and note that $\{C_{t+h}-C_t \geq 1\}=\{\Upsilon_{t} \leq t+h\} \in {\mathcal F}_{\Upsilon_t}$. Writing
$$\{C_{t+h}-C_t \geq 2\} = \{\Upsilon_t < t+h,\;\Upsilon_{\Upsilon_{t}} \leq t+h\}$$
and conditioning with respect to ${\mathcal F}_{\Upsilon_t}$, we first have
\begin{align*}
& \PP(C_{t+h}-C_t \geq 2) \\
= &\; \E\Big[\int_t^{t+h-\Upsilon_t}B\Big(\tfrac{xe^{\overline{{\mathcal V}}(b_{\vartheta_{C_t}})+s{\mathcal V}(s)}}{2^{C_t}}\Big)e^{-\int_0^s B\big(\frac{xe^{\overline{{\mathcal V}}(b_{\vartheta_{C_t}})+s'{\mathcal V}(s')}}{2^{C_t}}
\big)ds'}ds\,{\bf 1}_{\{\Upsilon_t < t+h\}}
\Big] \\
\leq &\;
h \sup_{y \leq x \exp (2e_{\max} t)} 
B(y)\,\PP(\Upsilon_t < t+h).
\end{align*}
In the same way, $\PP(\Upsilon_t < t+h) \lesssim h$ and the conclusion follows.
\subsection*{Acknowledgements}  
{\small The research of M. Doumic is partly supported by the Agence Nationale de la Recherche, Grant No. ANR-09-BLAN-0218 TOPPAZ. The research of M. Hoffmann is partly supported by the Agence Nationale de la Recherche, (Blanc SIMI 1 2011 project CALIBRATION). We are grateful to  V. Bansaye, A. Dalalyan, A. Marguet, E. Moulines, A. Olivier, B. Perthame, V. Rivoirard  and V.C. Tran for helpful discussion and comments.}

\bibliographystyle{plain}       
\bibliography{biblio}           
\end{document}